\documentclass[a4paper,11pt]{article}
\usepackage{geometry}
\usepackage[T1]{fontenc}
\usepackage[utf8]{inputenc}
\usepackage[english]{babel}
\usepackage{microtype}
\usepackage{braket}
\usepackage{amsmath}
\usepackage{amssymb}
\usepackage{amsthm, aliascnt}
\usepackage{mathtools}
\usepackage{mathrsfs}
\usepackage{xcolor}
\usepackage[hyphens]{url}
\usepackage{hyperref}
\usepackage{graphicx}
\usepackage{enumitem}
\usepackage{tikz}
\usepackage[maxbibnames=99]{biblatex}
\usepackage{csquotes}
\usetikzlibrary{patterns}
\usepackage{tikz}
\usepackage{esint}
\usepackage{mleftright}

\hypersetup{
                colorlinks=true,
                linkcolor={black},
                citecolor={black},
                breaklinks=true}

\theoremstyle{plain}
\newtheorem{teor}{Theorem}
\numberwithin{teor}{section}
\numberwithin{equation}{section}

\theoremstyle{definition}
\newaliascnt{defi}{teor}
\newtheorem{defi}[defi]{Definition}
\aliascntresetthe{defi}

\theoremstyle{plain}
\newaliascnt{lemma}{teor}
\newtheorem{lemma}[lemma]{Lemma}
\aliascntresetthe{lemma}
\theoremstyle{plain}
\newaliascnt{prop}{teor}
\newtheorem{prop}[prop]{Proposition}
\aliascntresetthe{prop}
\theoremstyle{plain}
\newaliascnt{conjecture}{teor}

\aliascntresetthe{conjecture}
\theoremstyle{plain}
\newaliascnt{cor}{teor}

\aliascntresetthe{cor}
\theoremstyle{definition}
\newaliascnt{ex}{teor}

\aliascntresetthe{ex}
\theoremstyle{definition}
\newaliascnt{oss}{teor}
\newtheorem{oss}[oss]{Remark}
\aliascntresetthe{oss}
\theoremstyle{plain}

\DeclarePairedDelimiter{\abs}{\lvert}{\rvert}
\DeclarePairedDelimiter{\norma}{\lVert}{\rVert}
\DeclareMathOperator{\sbv}{SBV}

\newcommand{\R}{\mathbb{R}}

\newcommand{\Hn}{\mathcal{H}^{n-1}}
\newcommand{\eps}{\varepsilon}

\DeclareMathOperator{\divv}{div}

\title{On the classical Reinforcement problem and Optimisation}
\author{E. Cristoforoni\textsuperscript{\textsc{(a)}}, C. Nitsch\textsuperscript{\textsc{(b)}}, C. Trombetti\textsuperscript{\textsc{(b)}} }
\date{}

\newcommand{\Addresses}{{
 \footnotesize 

\textsc{(a) Dipartimento di Matematica ``Federigo Enriques'', Universit\'a degli Studi di Milano La Statale, Via Saldini 50 20123 Milano Italia.}\par\nopagebreak 
 
 \medskip 
 
 \textit{E-mail address}, E.~Cristoforoni: \texttt{emanuele.cristoforoni@unimi.it} 

  \bigskip 
 
 \textsc{(b) Dipartimento di Matematica e Applicazioni ``R. Caccioppoli'', Universit\`a degli Studi di Napoli Federico II, Via Cintia, Complesso Universitario Monte S. Angelo, 80126 Napoli, Italia.}\par\nopagebreak

 \medskip 
 
 \textit{E-mail address}, C.~Nitsch: \texttt{c.nitsch@unina.it}

  \medskip 
 
 \textit{E-mail address}, C.~Trombetti: \texttt{cristina@unina.it} 
}}

\begin{document}

\maketitle

\Addresses

\begin{abstract}
    In the present survey, we consider the classical reinforcement problem for elliptic boundary value problems originally studied by Sanchez-Palencia in 1969. We focus on the seminar papers by Brezis, Caffarelli, \& Friedman, and by Acerbi \& Buttazzo, and discuss the related optimisation problems proposed by Friedman and by Buttazzo.
\end{abstract}

\begin{center}
    \begin{minipage}{.75\textwidth}
        \small
        \tableofcontents
    \end{minipage}
\end{center}

\section{Introduction}
Many problems in mathematical physics, such as in elasticity, electrostatics, or thermal conduction, are governed by elliptic (or parabolic) equations or variational inequalities whose coefficients have a jump across a smooth surface $S$. These models describe the behaviour of a physical quantity, $u$, in a medium, $D$, composed of two phases, $\Omega$ and $\Sigma$. The two phases come into contact across the smooth interface $S=\partial\Omega\cap\partial\Sigma$, and have different physical properties, such as elasticity or diffusivity coefficients, described by positive parameters $k_\Omega$ and $k_\Sigma$. Let $f$ be a given function. The prototype equation one should think of is the following 
\begin{equation}\label{eq0}
    \begin{cases}
        -k_\Omega\Delta u = f &\text{in } \Omega,\\[5 pt]
               -k_\Sigma\Delta u = f &\text{in } \Omega,\\[5 pt]
         u|_\Omega=u|_\Sigma &\text{on } S,\\[5 pt]
         k_\Omega \partial_\nu u|_\Omega = k_\Sigma \partial_\nu u|_\Sigma &\text{on }S\setminus\partial D,\\[5 pt]
         u=0 &\text{ on } \partial D,
    \end{cases}   
\end{equation}
where $\nu$ is normal to $S$. The condition 
\[u|_\Omega=u|_\Sigma\]
across the interface $S$, is a continuity condition which models the fact that the two phases $\Omega$ and $\Sigma$ are in perfect contact with each other, while the condition 
\begin{equation}\label{diffractionBC} k_\Omega \partial_\nu u|_\Omega = k_\Sigma \partial_\nu u|_\Sigma \quad\text{on }S\setminus\partial D,\end{equation}
 models the continuity of the flux across the interface. Condition \eqref{diffractionBC} is known as transmission, or diffraction, condition.\medskip
 
 More generally, let $A_1, A_2$ be a positive definite symmetric matrices. We can consider the following diffraction problem
 \begin{equation}\label{eqA}
 \begin{cases}
      -k_\Omega\divv(A_1 \nabla u) = f &\text{in } \Omega,\\[5 pt]
        -k_\Sigma\divv(A_2 \nabla u) = f &\text{in } \Sigma,\\[5 pt]
         u|_\Omega=u|_\Sigma &\text{on } S\setminus\partial D,\\[5 pt]
         k_\Omega \partial_{\nu_{A_1}} u|_\Omega = k_\Sigma \partial_{\nu_{A_2}} u|_\Sigma &\text{on }S\setminus\partial D,\\[5 pt]
         u=0 &\text{ on } \partial D,
 \end{cases}
 \end{equation}
 and $\partial_{\nu_A}u$ is the derivative in the conormal direction $\nu_A=A\nu$.
A weak solution to \eqref{eq0} is simply a function $u\in H^1_0(D)$ such that 
 \[k_\Omega \int_\Omega A_1\nabla u\cdot \nabla \varphi\,dx+k_\Sigma \int_\Sigma A_2\nabla u\cdot \nabla \varphi\,dx=\int_D f\varphi\,dx,\]
 for every $\varphi\in H^1_0(D)$, or, equivalently, the minimiser in $H^1_0(D)$ of the energy 
 \[k_\Omega\int_\Omega A_1\nabla v\cdot\nabla v\,dx+k_\Sigma\int_\Sigma A_2\nabla v\cdot\nabla v\,dx-2\int_D fv\,dx.\]
 Existence and regularity for the solutions to \eqref{eqA} have been widely studied in the '60s, in particular, if $A_1\in C^{1,\gamma}(\bar{\Omega}), A_2\in C^{1,\gamma}(\bar{\Sigma})$, $f\in C^{0,\gamma}(\bar{\Omega})\cap C^{0,\gamma}(\bar{\Sigma})$, and both the interface $S$ and the outer boundary $\partial D$ are of class $C^{1,1}$, then, the solution $u$, restricted to each phase, is $C^2$ with first derivative continuos up to the boundary with possibly the exception of the points in $S\cap\partial D$. We refer to \cite[\S 3.16]{LU68} for the precise statement (see also \cite{O61, LRU64, LU68}).\medskip

 Of particular interest in many physical models is the case in which the thickness of one of the phases, say $\Sigma$, is much smaller than the characteristic size of $D$, and the constants $ k_\Omega$ and $ k_\Sigma$ differ widely. Mathematically, this consists of the study of the \emph{asymptotic behaviour of the solutions} of \eqref{eqA} in the \emph{limit} as both the \emph{thickness} of $\Sigma$ goes to zero and the \emph{ellipticity constant} $k_\Sigma$ goes either to \emph{zero} or to \emph{infinity}.
 In \emph{elasticity} (or elasto-plasticity), the magnitude of the ellipticity constant is related to the material's hardness: a small ellipticity constant corresponds to a \emph{hard} material, whereas a large one corresponds to a \emph{soft} material. In this context, $D$ often represents the cross-section of a torqued beam composed of two (or more) phases. Then, $\Sigma$ is seen as a small quantity of \emph{extremely hard} (or \emph{extremely soft}) material added to the remaining material, $\Omega$, to \emph{reinforce} the beam. Given the previous interpretation, these types of problems are often referred to as \emph{reinforcement problems}, sometimes further divided into \emph{interior} reinforcement and \emph{boundary} reinforcement. Of particular significance in the case of boundary reinforcement is also the point of view of thermal conduction, in which the set $\Omega$ represents a conductor surrounded by an extremely thin layer, $\Sigma$, of extremely good (or extremely bad) \emph{insulating material}, in such framework, the solution $u$ in $D$ represents the (steady-state) \emph{temperature} of the configuration.\medskip
 
 The survey is structured as follows. 
 We start with a brief history of the problem, outlining the original result by Sanchez-Palencia. We then, in \autoref{sbsBCF} and \autoref{sbsAB}, prove the celebrated results by Brezis, Caffarelli \& Friedman \cite{BCF80} and by Acerbi \& Buttazzo \cite{AB86} respectively. In \autoref{sbsnew}, we give an overview of subsequent developments with particular attention to the case of different boundary conditions. Finally, in \autoref{sopt} we discuss the optimisation problem naturally arising from the limit equations.
 
 \section{A brief history of the problem}
 
 Consider the problem of describing the temperature of a system, D, composed of a thin metal sheet in contact with a relatively poor conductor, such as soil or insulating material. This problem is briefly described in the second edition of the book \textit{Conduction of Heat in Solids} by Carslaw, prepared by Jaeger in 1959 (see \cite{CJ59}). Let $\omega\subset\R^3\cap\set{z=0}$ be a bounded planar set and let $\Sigma_\eps=\omega\times(0,\eps)$ represent a thin metal sheet of constant thickness $\eps>0$, whose faces are both covered by insulating material $\Omega_\eps=D\setminus\Sigma_\eps$ so that the interface $S_\eps$ consists exactly of the set $\omega\times\set{0,\eps}$ (see \autoref{fig1}).
 \begin{figure}
     \centering
     \includegraphics[width=0.7\linewidth]{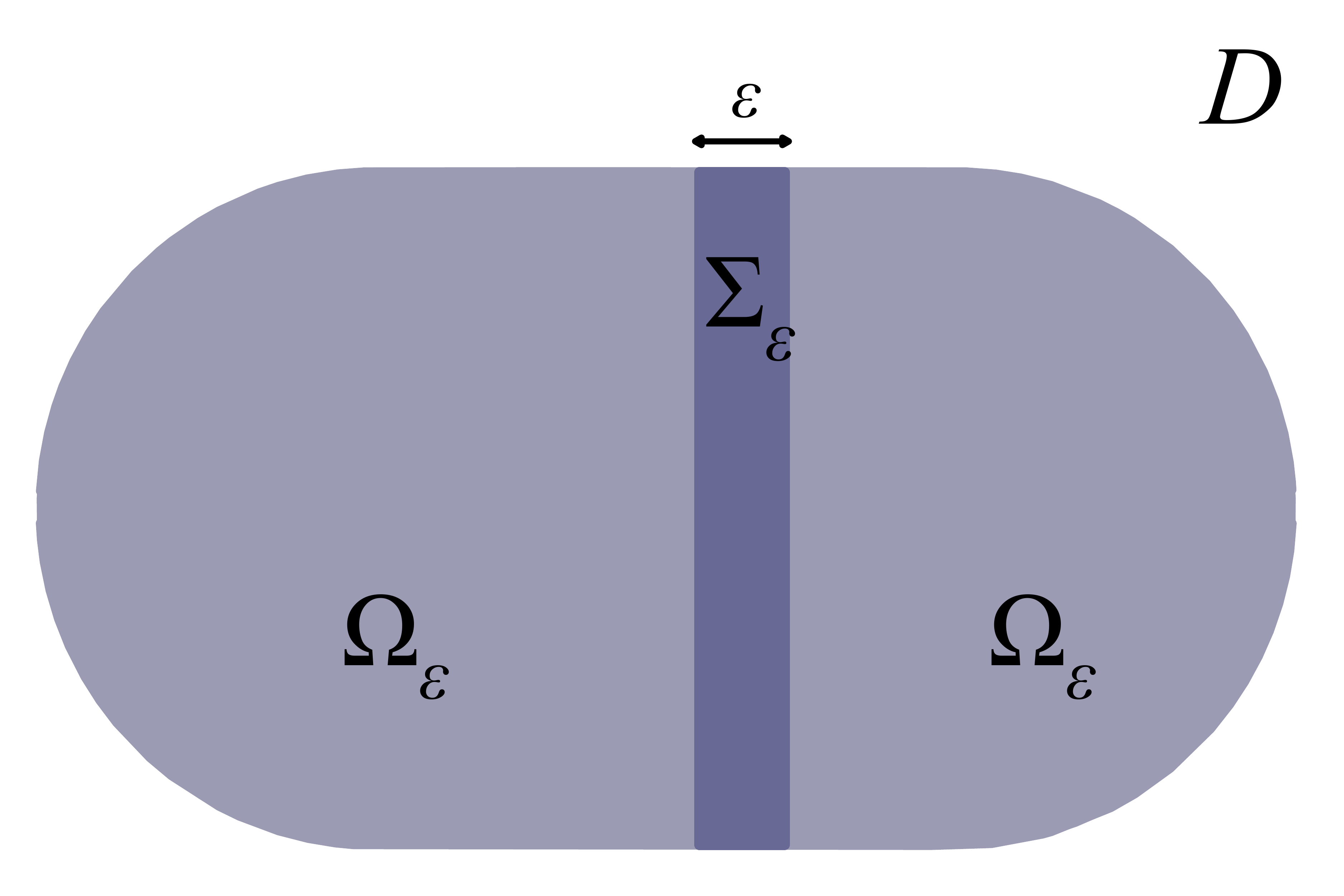}
     \caption{Thin metal sheet $\Sigma_\eps$ covered by insulating material $\Omega_\eps=D\setminus\Sigma_\eps$.}
     \label{fig1}
 \end{figure}
For simplicity' sake, let $k_\eps=k_{\Sigma_\eps}$ and $k=k_{\Omega_\eps}$. The steady-state temperature, $u_\eps$, will be a solution to 
\begin{equation}\label{ex1}    \begin{cases}
        -k\Delta u_{\eps} = f &\text{in } \Omega_\eps\\[5 pt]
         -k_\eps\Delta u_{\eps} = 0 &\text{in } \Sigma_\eps\\[5 pt]
         u_\eps|_{\Omega_\eps}=u_\eps|_{\Sigma_\eps} &\text{on } S_\eps,\\[5 pt]
         k \partial_\nu u_\eps|_{\Omega_\eps} = k_\eps \partial_\nu u_\eps|_{\Sigma_\eps} &\text{on }S_\eps\setminus\partial D,\\[5 pt]
         u_\eps=0 &\text{ on } \partial D.
    \end{cases}  \end{equation}
    In the book, the assumption made is that the metal sheet is so thin that the temperature in the sheet can be taken to be constant across its thickness. Then, considering the hat balance in the sheet, they simplify equation \eqref{ex1} by substituting the equation in $\Sigma_\eps$ and the transmission condition on $S_\eps$, with an effective boundary condition, of Wentzel type, relating the "two-dimensional" Laplacian on the flat surface of the sheet with the outward derivative in the direction orthogonal to it, that is 
 \[    \begin{cases}
        -k\Delta u_{\eps} = f &\text{in } \Omega_\eps,\\[5 pt]
               \partial_z u_{\eps} = 0 &\text{in } \Sigma_\eps,\\[5 pt]
         u_\eps|_{\Omega_\eps}=u_\eps|_{\Sigma_\eps} &\text{on } S_\eps,\\[5 pt]
         \eps k_\eps \left(\partial^2_{xx}u_\eps|_{\Sigma_\eps}+\partial^2_{yy}u_\eps|_{\Sigma_\eps}\right)  = - k \partial_z u_\eps|_{\Omega_\eps} &\text{on }\omega\times\set{0},\\[5 pt]
             \eps k_\eps \left(\partial^2_{xx}u_\eps|_{\Sigma_\eps}+\partial^2_{yy}u_\eps|_{\Sigma_\eps}\right)  = k \partial_z u_\eps|_{\Omega_\eps} &\text{on }\omega\times\set{\eps},\\[5 pt]
         u_\eps=0 &\text{ on } \partial D.
    \end{cases}  \]
    Notice that, while $\eps$ is assumed to be small, the ratio $k_\eps/k$, which represents the ratio of the thermal conductivities, is, by the physical assumptions on the problem, really high.\medskip

 To our knowledge, the first rigorous mathematical proof of this phenomenon can be found in the work of Sanchez-Palencia \cite{SP69}, published in 1969 (see also the paper by Hung and Sanchez-Palencia \cite{HSP74}, published in 1974). In the paper, the author studies the limit behaviour of the function $u_\eps$, solution to \eqref{ex1}, in the limit as $\eps$ goes to zero and $k_\eps$ diverges to infinity, the limit temperature then will be the solution to a boundary value problem depending on the limit, $\eta$, of the quantity
 $\eps k_\eps$. In particular, when $\eta$ is strictly positive and finite, the function $u_\eps$ converges to the solution to the following boundary value problem
 \[\begin{cases}
      -k\Delta u = f &\text{in } D\setminus\left(\omega\times\set{0}\right),\\[5 pt]
               u|_{\set{z>0}} = u|_{\set{z<0}} &\text{on }\omega\times\set{0},\\[5 pt]

        -\eta \left(\partial^2_{xx}u+\partial^2_{yy}u\right)  = k\left(\partial_z u|_{\set{z>0}}-\partial_z u|_{\set{z<0}}\right) &\text{on }\omega\times\set{0},\\[5 pt]
        \partial_\nu u = 0 &\text{on }\partial\omega\times\set{0},\\[5 pt]
         u=0 &\text{ on } \partial D.
 \end{cases}\]
When $\eta=0$, the effect of the plate becomes negligible, and the limit problem reduces to the usual Poisson problem
 \[\begin{cases}
     -k\Delta u = f &\text{in }D,\\[5 pt]
     u=0 &\text{on }\partial D.
 \end{cases}\]
On the other hand, if $\eta=+\infty$, in the limit, function $u$ will be forced to be constant in $\omega\times\set{0}$  and the limit boundary value problem will be the following 
 \[\begin{cases}
      -k\Delta u = f &\text{in } D\setminus\left(\omega\times\set{0}\right),\\[5 pt]
               u|_{\set{z>0}} =c= u|_{\set{z<0}}&\text{on }\omega\times\set{0},\\[5 pt]
         u=0 &\text{ on } \partial D,
 \end{cases}\]
 with the extra integral condition
 \[         \displaystyle\int_\omega \partial_z u|_{\set{z>0}} = \int_\omega\partial_z u|_{\set{z<0}} .\]
 As noted by the author, from the point of view of electrostatics, the case $\eta=+\infty$ corresponds to the definition, given in \cite{P56} (see also \cite[\S 4.16]{P63}), of an infinitely conducting and infinitely thin plate.
 \medskip

In 1970, Sanchez-Palencia \cite{SP70} (see also \cite{SP74}) studied the opposite case, namely, a thin sheet of insulating material in contact with a good conductor. Mathematically, this corresponds to the study of the limit behaviour of the function $u_\eps$, solution to \eqref{ex1}, in the case in which both $\eps$ and $k_\eps$ go to zero. In such a case, the limit problem will depend on the limit, $\alpha$, of the quantity $k_\eps/\eps$. When $\alpha$ is finite, the limit function will usually be discontinuous across $\omega\times\set{0}$, and will satisfy 
\[
\begin{cases}
      -k\Delta u = f &\text{in }  D\setminus\left(\omega\times\set{0}\right),\\[5 pt]
               \partial_z u|_{\set{z>0}} = \partial_z u|_{\set{z<0}} &\text{on }\omega\times\set{0},\\[5 pt]

        k \partial_z u  = \alpha\left(u|_{\set{z>0}} - u|_{\set{z<0}}\right) &\text{on }\omega\times\set{0},\\[5 pt]

         u=0 &\text{ on } \partial D.
 \end{cases}
\]
When $\alpha=+\infty$, the effect of the insulating sheet becomes negligible and the limit problem reduces to the usual Poisson problem in $D$. \medskip

Another interesting physical problem, described in the book by Carslaw and Jaeger, is the case of \emph{a thin surface skin of poor conductor}, $\Sigma_\eps$, surrounding a body $\Omega$. Let $\Sigma_\eps$ be the set of points at distance at most $\eps>0$ from $\Omega$, then, if $\eps$ is sufficiently small, the temperature $u_\eps$, solution to \eqref{ex1}, can be approximated by substituting the equation in $\Sigma_\eps$ with an effective boundary condition of Robin type, that is 
\[\begin{cases}
    -k\Delta u_{\eps} = f &\text{in } \Omega\\[5 pt]
    k\partial_\nu u_{\eps} + \frac{k_\eps}{\eps}u_\eps=0 &\text{on }\partial\Omega.
\end{cases}\]
This approximation is consistent with the results on boundary reinforcement by Caffarelli and Friedman \cite{CF80} in elasto-plasticity, as well as the results of Brezis, Caffarelli \& Friedman that we will discuss in the next section. In particular, if we denote by $\alpha$ the limit of the quantity $k_\eps/\eps$, when $\alpha>0$ is finite, the function $u_\eps$, solution to \eqref{ex1}, will converge to the solution of
\[\begin{cases}
    -k\Delta u = f &\text{in } \Omega\\[5 pt]
    k\partial_\nu u + \alpha u =0 &\text{on }\partial\Omega.
\end{cases}\]
As before, when $\alpha=+\infty$, the effect of the surface layer $\Sigma_\eps$ will become negligible. The case $\alpha=0$ may be a bit more involved and will be analysed in \autoref{alpha=0}.

\subsection{The result by Brezis, Caffarelli \& Friedman}\label{sbsBCF}

In this section, we discuss the results proved by Brezis, Caffarelli \& Friedman in \cite{BCF80}. In the paper, the authors prove convergence results for boundary and interior reinforcement problems for elliptic equations as well as the corresponding variational inequalities. As the techniques used in the various results are essentially the same, we will focus on the boundary reinforcement for elliptic equations.\medskip

Let $\Omega\subset\R^n$ be a bounded open set with $C^{1,1}$ boundary, let $h\colon\partial\Omega\to\R$ be a strictly positive Lipschitz function. For every $\eps>0$ let
\[\Sigma_\eps=\Set{\sigma+t\nu(\sigma)\colon \sigma\in\partial\Omega,\,t\in(0,\eps h(\sigma))},\]
where $\nu$ denotes the outward unit normal vector to $\partial\Omega$, an let $D_\eps=\bar{\Omega}\cup\Sigma_\eps$ (see \autoref{fig2}).
\begin{figure}
    \centering
    \includegraphics[width=0.7\linewidth]{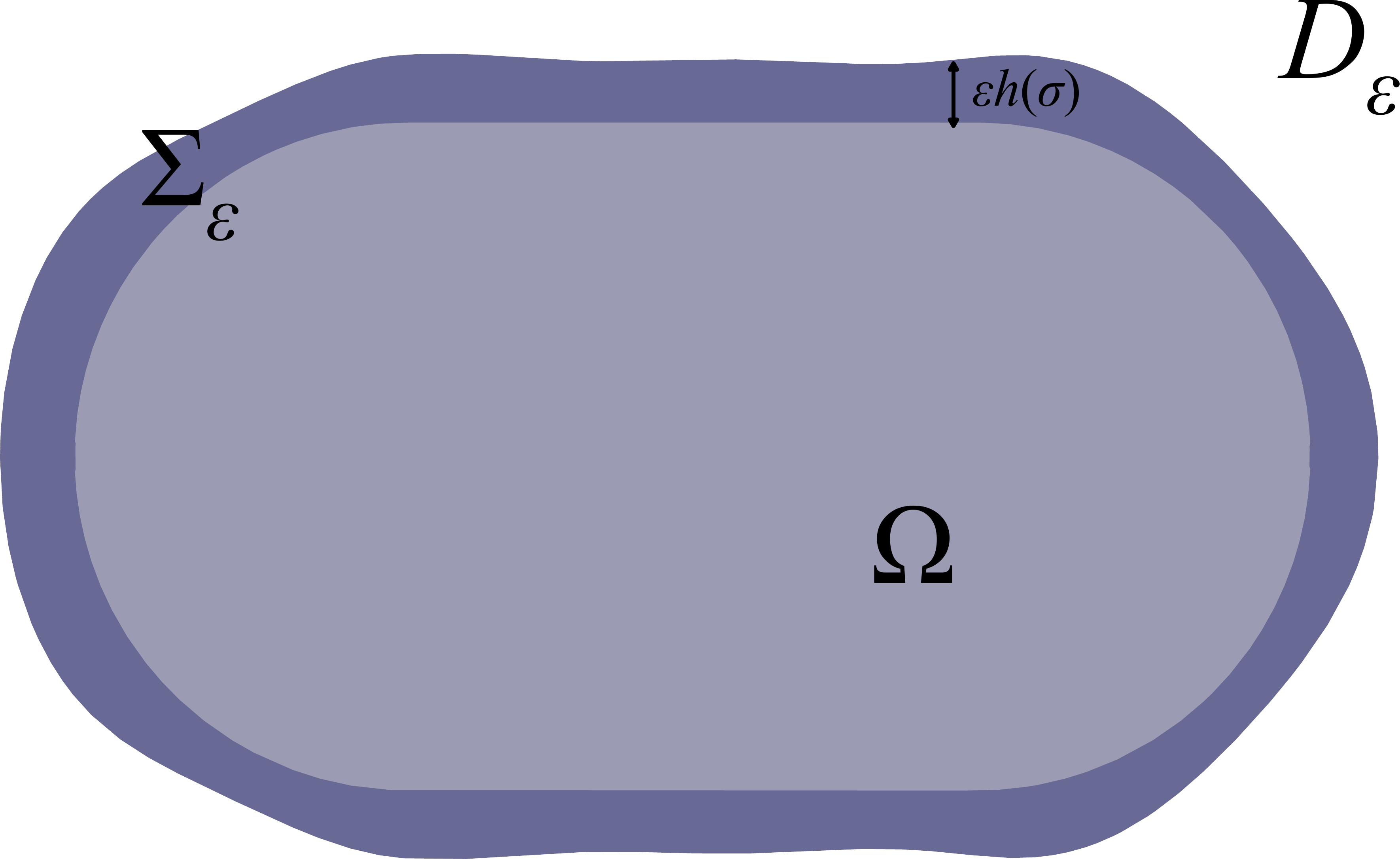}
    \caption{Reinforcement of $\Omega$ with a thin set of variable thickness $\Sigma_\eps$.}
    \label{fig2}
\end{figure}
Our regularity assumptions on $\Omega$ ensure us that, for every $\eps<\bar{\eps}$ sufficiently small, the map
\[(\sigma,t)\to\sigma+t\nu(\sigma)\]
is invertible on $\Sigma_\eps$, that is for every $x\in\Sigma_\eps$ there exist  unique $\sigma(x)\in\partial\Omega$ and $t(x)\in(0,\eps h(\sigma))$, such that 
\[x=\sigma(x)+t(x)\nu(\sigma(x)),\]
moreover $t(x)$ coincides with the distance, $d(x)$, of $x$ from $\Omega$, and $\sigma(x)$ is the orthogonal projection on $\partial\Omega$.\medskip

Let $A_1,A_2$ be bounded and uniformly elliptic matrices with coefficients in $C^{1,\gamma}(\bar{\Omega})$ and $ C^{1,\gamma}(\bar{\Sigma}_{\bar{\eps}})$, for some $\gamma>0$,  respectively and denote by $\nu_{A_1}=A_1 \nu$ and $\nu_{A_2}=A_2 \nu$ the conormal vector to $\partial\Omega$ from inside and outside $\Omega$ respectively. Let $f\colon\bar{D}_{\bar{\eps}}\to\R$ be a function in $C^{1,\gamma}(\bar{\Omega})\cup C^{1,\gamma}(\bar{\Sigma}_{\bar{\eps}})$, and let $k,k_\eps$ be two positive numbers with
\[\lim_{\eps\to0}k_\eps=0.\]
For every $\eps\in(0,\bar{\eps})$, let $u_\eps$ be the solution to the diffraction problem 
\begin{equation}\label{BCF}
    \begin{cases}
        -k\divv(A_1\nabla u_\eps)=f &\text{in }\Omega,\\[5pt]
        -k_\eps\divv(A_2\nabla u_\eps)=f &\text{in }\Sigma_\eps,\\[5pt]
        u_\eps|_\Omega = u_\eps|_{\Sigma_\eps} &\text{on }\partial\Omega,\\[5pt]
        k\partial_{\nu_{A_1}}u_\eps|_\Omega = k_\eps\partial_{\nu_{A_2}}u_\eps|_{\Sigma_\eps} &\text{on }\partial\Omega,\\[5pt]
        u_\eps = 0 &\text{on }\partial D_\eps.
    \end{cases}
\end{equation}
We will study the limit of the function $u_\eps$, as $\eps$ goes to zero, following the original proof by Brezis, Caffarelli, \& Friedman. As mentioned in the introduction, such a limit will depend on the limit
\[\alpha=\lim_{\eps\to0}\dfrac{k_\eps}{\eps}\in[0,+\infty].\]

\begin{oss}\label{conormalgraph}
We remark that in the original paper \cite{BCF80}, the set $\Sigma_\eps$ (denoted as $\Omega_1$ in the paper) is not described as the \emph{normal} graph of the function $h$, but rather as the \emph{conormal} graph instead. This difference will result in a slightly more involved boundary condition in the limit problem; on the other hand, it will give us a clearer distinction between the role of the geometry and that of the operator. By the ellipticity of $A_2$ we have that $\gamma_{A_2}=\nu_{A_2}\cdot\nu>c>0$ uniformly on $\partial\Omega$ and, by the regularity of $\Omega$ we can write $\Sigma_\eps$ as
\[\Sigma_\eps=\Set{\sigma+t\nu_{A_2}(\sigma)\colon \sigma\in\partial\Omega,\,t\in(0, h_\eps(\sigma))},\]
for some positive function $h_\eps$. Using the regularity of $\Omega$ and $h$, one may easily verify that 
\begin{equation}\label{gammaA}\lim_{\eps\to0}\dfrac{\eps h(\sigma)}{h_\eps (\sigma)}=\gamma_{A_2}(\sigma),\end{equation}
uniformly on $\partial\Omega$.
\end{oss}

\begin{oss}
    We remark that by a change of variable, we can rewrite the integrals on $\Sigma_\eps$ as
    \[\int_{\Sigma_\eps}v(x)\,dx=\int_{\partial\Omega}\int_0^{h_\eps(\sigma)}v(\sigma+t\nu_{A_2})\gamma_{A_2}(\sigma)(1+tq(t,\sigma))\,dt\,d\sigma,\]
    where $q$ is a bounded function. Moreover, by \eqref{gammaA}, if $\eps$ is sufficiently small we have 
    \begin{equation}\label{eqivint}\dfrac{1}{C}\int_{\partial\Omega}\int_0^{h_\eps(\sigma)}v(\sigma+t\nu_{A_2})\,dt\,d\sigma\le\int_{\Sigma_\eps}v(x)\,dx\le C\int_{\partial\Omega}\int_0^{h_\eps(\sigma)}v(\sigma+t\nu_{A_2})\,dt\,d\sigma.\end{equation}
\end{oss}
from the previous equation, if $v\in H^1(\Sigma_\eps)$ with zero trace on $\partial D_\eps$ we immediately have that 
\begin{equation}\label{epscontrol}\int_{\partial\Omega}v^2\,d\sigma\le \eps C \int_{\Sigma_\eps} \abs{\nabla v}^2\,dx.\end{equation}
Indeed 
\[v(\sigma)=-\int_0^{h_\eps} \partial_{\nu_{A_2}}v(\sigma+t\nu_{A_2})\,dt,\]
so that 
\[\int_{\partial\Omega} v^2\,d\sigma\le \int_{\partial\Omega}h_\eps(\sigma)\int_0^{h_\eps}\abs{\nu_{A_2}}^2\abs{\nabla v}^2\,dt\,d\sigma\le \eps C \int_{\Sigma_\eps} \abs{\nabla v}^2\,dx.\]

The main tool in the study of the limit behaviour of $u_\eps$ are the following $H^2$-estimates.

\begin{lemma}\label{BCF_lemma1}
    Let $u_\eps$ be the solution to \eqref{BCF}. There exists a positive constant $C>0$, such that, for every $\eps$ sufficiently small, the following estimates hold
    \begin{equation}\label{L2}
        \int_{D_\eps} u_\eps^2\,dx\le C\left(1+\dfrac{\eps^2}{k_\eps^2}\right),
    \end{equation}
    \begin{equation}\label{H1}
        \int_\Omega \abs{\nabla u_\eps}^2\,dx+k_\eps\int_{\Sigma_\eps}\abs{\nabla u_\eps}^2\,dx\le C\left(1+\dfrac{\eps}{k_\eps}\right),
    \end{equation}
        \begin{equation}\label{H2}
        \int_\Omega \abs{D^2 u_\eps}^2\,dx+k_\eps\int_{\Sigma_\eps}\abs{D^2 u_\eps}^2\,dx\le C\left(1+\dfrac{\eps}{k_\eps}\right),
    \end{equation}
    where $D^2 u_\eps$ is the Hessian of $u_\eps$.
\end{lemma}
The previous lemma is based on classic elliptic estimates for diffraction problems in which we keep track of the ellipticity constant $k_\eps$ in $\Sigma_\eps$.

\subsubsection*{Case 1. $\alpha\in(0,+\infty]$}
\begin{teor}\label{BCFTh1}
    Let $u_\eps$ be the solution to equation \eqref{BCF}. If 
    \[\alpha=\lim_{\eps\to0}\dfrac{k_\eps}{\eps}\in(0,+\infty],\]
    then $u_\eps$ converges, uniformly in compact subsets of $\Omega$, to the function $u$ solution to
    \begin{equation}\label{limcase1}\begin{cases}
    -k\divv(A_1\nabla u)=f &\text{in }\Omega,\\[5pt]
    \dfrac{hk}{\alpha \gamma_{A_2}}\partial_{\nu_{A_1}}u+u=0 &\text{on }\partial\Omega.
    \end{cases}\end{equation}
\end{teor}
\begin{proof}
      Let $u$ be the solution to \eqref{limcase1} and let $w_\eps=u_\eps-u$, then in $\Omega$ $w_\eps$ satisfies 
      \begin{equation}\label{eqdiff1}\begin{cases}
    -k\divv(A_1\nabla w_\eps)=0 &\text{in }\Omega,\\[5pt]
    \dfrac{hk}{\alpha \gamma_{A_2}}\partial_{\nu_{A_1}}w_\eps+w_\eps = R_\eps&\text{on }\partial\Omega,
    \end{cases}\end{equation}
    where \[R_\eps=\dfrac{hk}{\alpha\gamma_{A_2}}\partial_{\nu_{A_1}}u_\eps+u_\eps.\]
    If we prove that $R_\eps$ converges to zero in $L^2(\Omega)$ as $\eps$ goes to zero, we can use the Green function representation of the solution of \eqref{eqdiff1} to prove that $w_\eps$ converges uniformly to zero in compact subsets of $\Omega$.\medskip
     
    Recall that, by our assumptions, the functions $u_\eps|_\Omega$ and $u_\eps|_{\Sigma_\eps}$ are $C^2$ with continuous first derivatives up to $\partial\Omega$. Let $\sigma\in\partial\Omega$, using the transmission conditions on $\partial\Omega$, we have
    \begin{equation}\label{transmission}u_\eps(\sigma)+\dfrac{ k}{k_\eps}h_\eps(\sigma)\partial_{\nu_{A_1}}u_\eps|_\Omega(\sigma)=u_\eps(\sigma)+h_\eps(\sigma) \partial_{\nu_{A_2}}u_\eps|_{\Sigma_\eps}(\sigma).\end{equation}
    Notice that, by \eqref{gammaA},  
     \[\lim_{\eps\to0}\dfrac{h_\eps(\sigma)k}{k_\eps}=\dfrac{h(\sigma)k}{\alpha \gamma_{A_2}(\sigma)},\]
     uniformly on $\partial\Omega$.\medskip
     
    From the Dirichlet boundary condition, we have
    \[u_\eps(\sigma)=u_\eps(\sigma)-u_\eps(\sigma+h_\eps\nu_{A_2})=-\int_0^{h_\eps(\sigma)}\partial_{\nu_{A_2}}u_\eps(\sigma+t\nu_{A_2})\,dt,\]
    while
    \[\partial_{\nu_{A_2}}u_\eps(\sigma+t\nu_{A_2})=\int_0^t\partial^2_{\nu_{A_2}^2}u_\eps(\sigma+s\nu_{A_2})\,ds+\partial_{\nu_{A_2}}u_\eps|_{\sigma_\eps}(\sigma).\]
    Substituting in the previous equation, we get
    \[u_\eps(\sigma)=-h_\eps(\sigma)\partial_{\nu_{A_2}}u_\eps|_{\sigma_\eps}(\sigma)-\int_0^{h_\eps(\sigma)}\int_0^t\partial^2_{\nu_{A_2}^2}u_\eps(\sigma+s\nu_{A_2})\,ds\,dt,\]
    so that, substituting the formula for $u_\eps(\sigma)$, in the right-hand side of \eqref{transmission}, we have
    \[u_\eps(\sigma)+\dfrac{ k}{k_\eps}h_\eps(\sigma)\partial_{\nu_{A_1}}u_\eps|_\Omega(\sigma)=-\int_0^{h_\eps(\sigma)}\int_0^t\partial^2_{\nu_{A_2}^2}u_\eps(\sigma+s\nu_{A_2})\,ds\,dt.\]
    Then if we take take square integral on $\partial\Omega$ using the change of variable \eqref{eqivint}, and use the $H^2$-estimate \eqref{H2}, we have
    \begin{equation}\label{bcest}
        \int_{\partial\Omega}\abs*{u_\eps+\dfrac{ k}{k_\eps}h_\eps\partial_{\nu_{A_1}}u_\eps|_\Omega}^2\,d\sigma\le C\eps^3\int_{\Sigma_\eps}\abs{D^2 u_\eps}^2\,dx
        \le C\eps^2\dfrac{\eps}{k_\eps}\left(\dfrac{\eps}{k_\eps}+C\right),
    \end{equation}
    by the assumption on $\alpha$, $\eps/k_\eps$ is bounded so that
    \[\int_{\partial\Omega}\abs*{u_\eps+\dfrac{ k}{k_\eps}h_\eps\partial_{\nu_{A_1}}u_\eps|_\Omega}^2\,d\sigma\le C\eps^2.\]
    On the other hand, by standard Sobolev inequalities and using estimates \eqref{H1} and \eqref{H2}, we have 
    \[\int_{\partial\Omega}\abs*{\partial_{\nu_{A_1}} u_\eps}^2\,d\sigma\le C\left(\int_\Omega \abs{\nabla u_\eps}^2\,dx+\int_\Omega \abs{D^2 u_\eps}^2\,dx\right)\le C.\]
    Finally, we have 
    \[\begin{split}\int_{\partial\Omega}\abs*{u_\eps+\dfrac{hk}{\alpha \gamma_{A_2}}\partial_{\nu_{A_1}}u_\eps}^2\,d\sigma&\le \int_{\partial\Omega}\abs*{u_\eps+\dfrac{ k}{k_\eps}h_\eps\partial_{\nu_{A_1}}u_\eps|_\Omega}^2\,d\sigma+\norma*{\dfrac{hk}{\alpha\gamma_{A_2}}-\dfrac{h_\eps k}{k_\eps}}_\infty\int_{\partial\Omega}\abs*{\partial_{\nu_{A_1}} u_\eps}^2\,d\sigma\\[10 pt]
    &\le C\left(\eps^2+\norma*{\dfrac{hk}{\alpha\gamma_{A_2}}-\dfrac{h_\eps k}{k_\eps}}_\infty\right),
    \end{split}\]
    that is, $R_\eps$ converges to zero in $L^2(\Omega)$, which concludes the proof.
\end{proof}

\subsubsection*{Case 2. $\alpha=0$}

In the following, we assume $\Omega$ to be connected

\begin{teor}\label{alpha=0}
    Let $u_\eps$ be the solution to equation \eqref{BCF}, and let $(u_\eps)_{\Omega}$ be its integral mean on $\Omega$. If 
    \[\alpha=\lim_{\eps\to0}\dfrac{k_\eps}{\eps}=0,\]
    then $u_\eps-(u_\eps)_{\Omega}$ converges, uniformly in compact subsets of $\Omega$, to the function $u$ solution to 
    \begin{equation}\label{limcase2}\begin{cases}
    -k\divv(A_1\nabla u)=f &\text{in }\Omega,\\[5pt]
    \dfrac{hk}{\gamma_{A_2}}\partial_{\nu_{A_1}}u=-\zeta &\text{on }\partial\Omega,\\[7pt]
    (u)_\Omega=0,
    \end{cases}
    \end{equation}
    where
    \[\zeta=\dfrac{\displaystyle\int_\Omega f\,dx}{\displaystyle\int_{\partial\Omega}\dfrac{\gamma_{A_2}}{h}\,d\sigma}.\]
\end{teor}
\begin{proof}
    Notice that inequality \eqref{bcest} still holds, so that, multiplying both sides by $k_\eps^2/\eps^2$ we get
    \begin{equation}\label{bcest2}
        \int_{\partial\Omega}\abs*{\dfrac{k_\eps}{\eps}u_\eps+\dfrac{ h_\eps}{\eps}k\partial_{\nu_{A_1}}u_\eps|_\Omega}^2\,d\sigma\le C\eps^2.
    \end{equation}
    On the other hand, by \eqref{epscontrol} and \eqref{H1} we have
    \begin{equation}\label{trest0}
        \int_{\partial\Omega} \dfrac{k_\eps^2}{\eps^2} u_\eps^2\,d\sigma\le C \dfrac{k_\eps}{\eps}k_\eps \int_{\Sigma_\eps}\abs{\nabla u_\eps}^2\,dx\le C\left(1+\dfrac{k_\eps}{\eps}\right)\le C.
    \end{equation}
    Joining \eqref{bcest2} and \eqref{trest0} we have
    \begin{equation}\label{conbound}\int_{\partial\Omega} \left(\partial_{\nu_{A_1}}u_\eps|_\Omega\right)^2\,d\sigma\le C.\end{equation}
    Using the Green function representation for the Neumann boundary condition in $\Omega$, we have
    \[u_\eps(x)=\int_\Omega G(x,y)f(y)\,dy+k\int_{\partial\Omega}G(x,\sigma)\partial_{\nu_{A_1}}u_\eps(\sigma)\,d\sigma+(u_\eps)_\Omega,\] so that, using \eqref{conbound}, we have that $u_\eps-(u_\eps)_{\Omega}$ is uniformly bounded, in $L^\infty$, in compact subsets of $\Omega$.\medskip

    Fix $\delta>0$ sufficiently small so that the function $\sigma\in\partial\Omega\mapsto\sigma_\delta=\sigma-\delta\nu(\sigma)\in\Omega$ is injective. For every $\sigma\in\partial\Omega$ we have 
    \[u_\eps(\sigma)-u_\eps(\sigma_\delta)=\int_0^\delta \partial_\nu u_\eps(\sigma-s\nu(\sigma))\,ds,\]
    so that
    \[\int_{\partial\Omega} \abs{u_\eps(\sigma)-u_\eps(\sigma_\delta)}^2\,d\sigma\le C\int_\Omega\abs{\nabla u_\eps}^2\,dx\]
    hence, multiplying both sides by $k_\eps^2/\eps^2$ and using \eqref{H1}, we have
    \begin{equation}\label{delta}
        \int_{\partial\Omega} \abs*{\dfrac{k_\eps}{\eps}u_\eps(\sigma)-\dfrac{k_\eps}{\eps}u_\eps(\sigma_\delta)}^2\,d\sigma\le C\dfrac{k_\eps}{\eps}.
    \end{equation}
Fix $x_0\in\Omega$, by the uniform boundedness of $u_\eps-(u_\eps)_\Omega$ in compact subsets of $\Omega$, we have that there exists a constant $C$ such that
\[\abs{u_\eps(\sigma_\delta)-u_\eps(x_0)}\le C,\]
then, letting $\zeta_\eps=k_\eps u_\eps(x_0)/\eps$, we find that
\begin{equation}\label{zeta}\int_{\partial\Omega} \abs*{\dfrac{k_\eps}{\eps}u_\eps(\sigma)-\zeta_\eps}^2\,d\sigma\le C\dfrac{k_\eps}{\eps}.\end{equation}
Then, recalling that $h_\eps/\eps$ converges uniformly on $\partial\Omega$ to $h/\gamma_{A_2}$, and using \eqref{bcest2} and \eqref{zeta}, we have 
\[\int_{\partial\Omega}\abs*{k\partial_{\nu_{A_1}}u_\eps+\dfrac{\gamma_{A_2}}{h}\zeta_\eps}^2\,d\sigma\xrightarrow{\eps\to0}0,\]
from which we deduce that
\[\lim_{\eps\to0}\zeta_\eps \int_{\partial\Omega} \dfrac{\gamma_{A_2}}{h}\,d\sigma=-\lim_{\eps\to0}k\int_{\partial\Omega}\partial_{\nu_{A_1}}u_\eps\,d\sigma=\int_\Omega f\,dx,\]
that is $\zeta_\eps$ converges to $\zeta$ and 
\begin{equation}\label{bclim}\int_{\partial\Omega}\abs*{k\partial_{\nu_{A_1}}u_\eps+\dfrac{\gamma_{A_2}}{h}\zeta}^2\,d\sigma\xrightarrow{\eps\to0}0.\end{equation}
Let $u$ be the solution to \eqref{limcase2}, then using the Green function representation for the Neumann boundary condition in $\Omega$ of $u_\eps-(u_\eps)_\Omega$ and $u$, and using \eqref{bclim}, we have the assertion.
\end{proof}

\begin{oss}
    Notice that, by \eqref{conbound} and the Poincaré-Wirtinger inequality, $u_\eps-(u_\eps)_\Omega$ is equibounded in $H^1(\Omega)$, so that it also converges to $u$ weakly $H^1(\Omega)$. Furthermore, we notice that if 
    \[\int_\Omega f\,dx\ne0,\]
    then $(u_\eps)_\Omega$ diverges as $\eps$ goes to zero. Indeed, assume by contradiction that for a subsequence (not relabelled)  $(u_\eps)_\Omega$ is bounded. Then $u_\eps$ is bounded in $H^1(\Omega)$ and in particular in $L^2(\partial\Omega)$. Hence, by \eqref{bcest2} we would have that $\partial_{\nu_{A_1}}u_\eps$ converges to zero in $L^2(\partial\Omega)$ but, for every $\eps$
    \[k\int_{\partial\Omega}\partial_{\nu_{A_1}}u_\eps\,d\sigma=\int_\Omega f\,dx\ne0,\]
    which is a contradiction. 
\end{oss}

Friedman, in \cite{F80}, proved analogous results for the convergence of the principal eigenvalue of the operator.
Consider the eigenvalue problem 
\begin{equation}\label{eig}
        \begin{cases}
        -k\divv(A_1\nabla \varphi)= \lambda \varphi&\text{in }\Omega,\\[5pt]
        -k_\eps\divv(A_2\nabla \varphi)=\lambda \varphi&\text{in }\Sigma_\eps,\\[5pt]
        \varphi|_\Omega = \varphi|_{\Sigma_\eps} &\text{on }\partial\Omega,\\[5pt]
        k\partial_{\nu_{A_1}}\varphi|_\Omega = k_\eps\partial_{\nu_{A_2}}\varphi|_{\Sigma_\eps} &\text{on }\partial\Omega,\\[5pt]
        \varphi = 0 &\text{on }\partial D_\eps,
    \end{cases}
    \end{equation}
    and denote by $\lambda_\eps$ the principal eigenvalue and by $\varphi_\eps\ge0$ the corresponding $L^2(D_\eps)$-normalised eigenfunction.
The main results of \cite{F80} can be summarised in the following theorem.
\begin{teor}
   Let 
   \[\alpha=\lim_{\eps\to0} \dfrac{k_\eps}{\eps}.\]
   If $\alpha=0$, $\lambda_\eps$ converges to zero and
   \[\lim_{\eps\to0}\dfrac{\eps}{k_\eps}\lambda_\eps=\int_{\partial\Omega}\dfrac{\gamma_{A_2}}{h}\,d\sigma.\]
   If $\alpha\in(0,+\infty]$, consider the eigenvalue problem
      \[\begin{cases}
        -k\divv(A_1\nabla \varphi)= \lambda \varphi&\text{in }\Omega,\\[5pt]
        \dfrac{hk}{\alpha \gamma_{A_2}}\partial_{\nu_{A_1}}u+u=0&\text{on }\partial\Omega,
    \end{cases}\]
   and denote by $\lambda_0$ the principal eigenvalue and by $\varphi_0\ge0$ the corresponding $L^2(\Omega)$-normalised eigenfunction. Then
   \[\lim_{\eps\to0}\lambda_\eps=\lambda_0\]  and $\varphi_\eps$ converges to $\varphi_0$ weakly in $H^2(\Omega)$.
\end{teor}

\subsubsection*{Interior reinforcement}

We now give an overview of the result for the case of interior reinforcement.

Let $\Omega_1, D \subset\R^n$ be bounded open sets with  $C^{1,1}$ boundary. Let $h\colon\partial\Omega_1\to\R$ be a strictly positive Lipschitz function and, for every $\eps<\bar{\eps}$, let
\[\Sigma_\eps=\Set{\sigma+t\nu\colon\sigma\in\partial\Omega_1,\,t\in(0,\eps h(\sigma))}.\]
 Moreover, assume that $\bar{\Omega}_1\cup\Sigma_\eps$ is compactly contained in $D$, and denote by $\Omega_{2,\eps}=D\setminus\left(\overline{\Omega_1\cup\Sigma_\eps}\right)$ (see \autoref{fig3}). \begin{figure}
     \centering
     \includegraphics[width=0.7\linewidth]{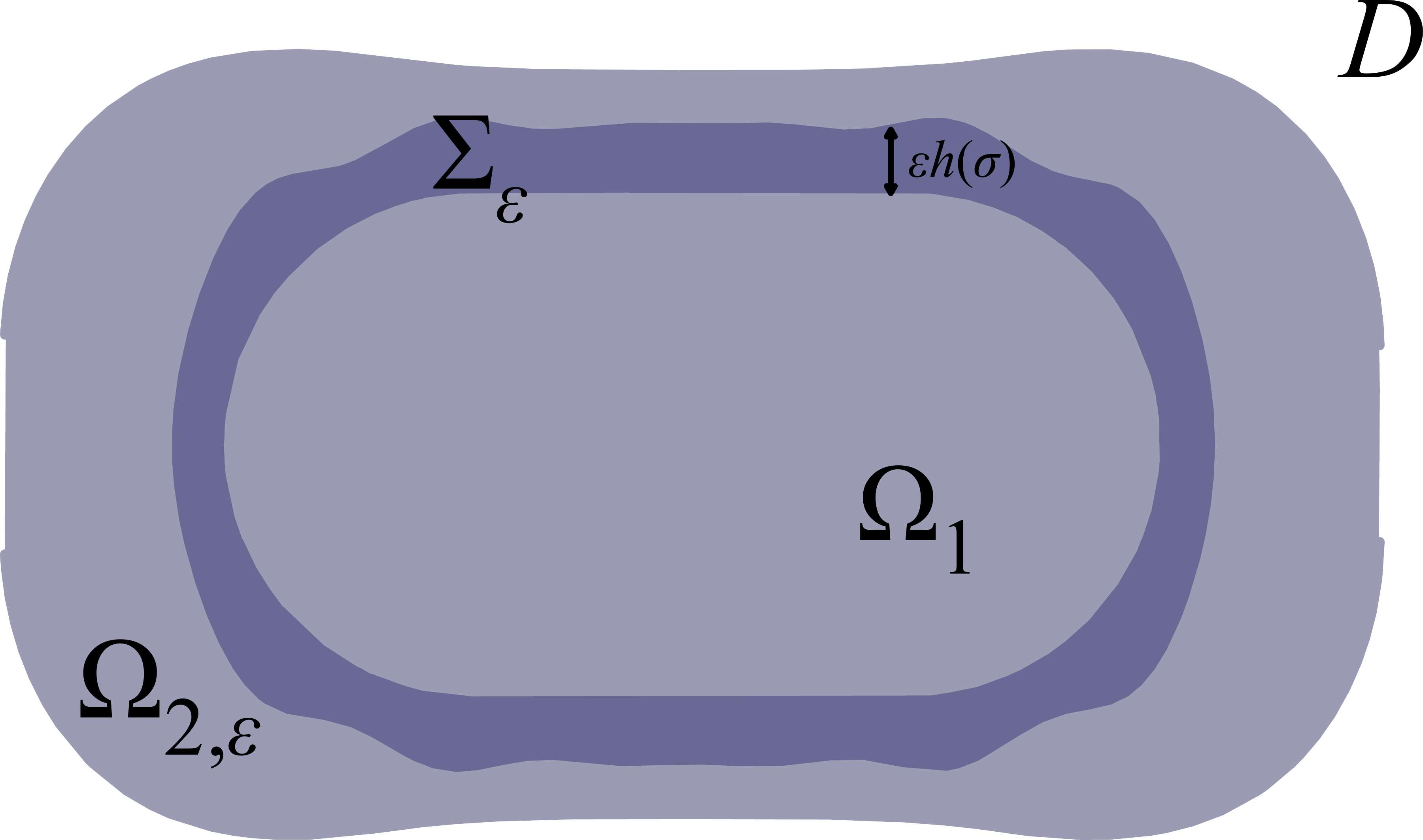}
     \caption{Interior reinforcement of $D$ with a thin set of variable thickness $\Sigma_\eps$. }
     \label{fig3}
 \end{figure}\medskip

 Let $A_1,A_2,A_3$ be bounded and uniformly elliptic matrices with smooth coefficients and let $f_1,f_2,f_3\colon \bar{D}\to\R $ be smooth functions. for every $\eps\in(0,\bar{\eps})$, let $v_\eps$ be the solution to the three-phase diffraction problem

 \begin{equation}\label{BCFint}
    \begin{cases}
        -k\divv(A_1\nabla v_\eps)=f_1 &\text{in }\Omega_1,\\[5pt]
        -k_\eps\divv(A_2\nabla v_\eps)=f_2 &\text{in }\Sigma_\eps,\\[5pt]
          -k\divv(A_3\nabla v_\eps)=f_3 &\text{in }\Omega_{2,\eps},\\[5pt]
        v_\eps|_\Omega = v_\eps|_{\Sigma_\eps} &\text{on }\partial\Omega_1,\\[5pt]
        k\partial_{\nu_{A_1}}v_\eps|_\Omega = k_\eps\partial_{\nu_{A_2}}v_\eps|_{\Sigma_\eps} &\text{on }\partial\Omega_1,\\[5pt]
        v_\eps|_{\Sigma_\eps} = v_\eps|_{\Omega_{2,\eps}} &\text{on }\partial\Sigma_\eps\setminus\partial\Omega_1,\\[5pt]
        k_\eps\partial_{\nu_{\eps,A_2}}v_\eps|_{\Sigma_\eps} = k\partial_{\nu_{\eps,A_3}}v_\eps|_{\Omega_{2,\eps}} &\text{on }\partial\Sigma_\eps\setminus\partial\Omega_1,\\[5pt]
        v_\eps = 0 &\text{on }\partial D,
    \end{cases}
\end{equation}
where $\nu_\eps$ is the normal direction to $\partial\Sigma_\eps\setminus\partial\Omega$, and $\nu_{\eps,A_2},\nu_{\eps,A_3}$ are the conormal directions $A_2 \nu_\eps$ and $A_3 \nu_\eps$.\medskip

As seen in the introduction to the classical results of Sanchez-Palencia, the limit function will usually be discontinuous; however, the transmission conditions for the conormal derivatives will carry through to the limit.  As for the case of boundary reinforcement, the limit problem will depend on the quantity
\[\alpha=\lim_{\eps\to0}\dfrac{k_\eps}{\eps}.\]
Let $\Omega_2=D\setminus\Omega_1$, we have the following theorem

\begin{teor}
Let $v_\eps$ be the solution to equation \eqref{BCFint}. If $\alpha>0$, then $v_\eps$ converges, uniformly in compact subsets of $\Omega_1$ and $\Omega_2$ to the function $v$ solution to
\[\begin{cases}
    -k\divv(A_1\nabla v)=f_1 &\text{in }\Omega_1,\\[5pt]
    -k\divv(A_3\nabla v)=f_3 &\text{in }\Omega_2,\\[5pt]
    \partial_{\nu_{A_1}} v = \partial_{\nu_{A_3}} v &\text{on }\partial\Omega_1,\\[5pt]
    v|_{\Omega_1}-v|_{\Omega_2}=-
    \dfrac{hk}{\alpha \gamma_{A_2}}\partial_{\nu_{A_1}}v &\text{on }\partial\Omega_1,\\[5pt]
    v=0 &\text{on }\partial D.\end{cases}
\]
\end{teor}

 \begin{proof}[Sketch of the proof]
 For every $\sigma\in\partial\Omega_1$ let $\sigma_\eps=\sigma+h_\eps(\sigma)\nu_{A_2}(\sigma)$ be the corresponding point on $\partial\Sigma_\eps\setminus\partial\Omega_1$ in the conormal direction $\nu_{A_2}$. The main idea of the proof is to estimate, with techniques similar to the ones of \autoref{BCFTh1}, the integrals on $\partial\Omega_1$ of the quantities 
 \[\abs{\partial_{\nu_{A_1}}v_\eps(\sigma)-\partial_{\nu_{A_2}}v_\eps(\sigma_\eps)},\]
 and
 \[\abs*{v_\eps(\sigma)-v_\eps(\sigma_\eps)+
    \dfrac{h(\sigma)k}{\alpha \gamma_{A_2}(\sigma)}\partial_{\nu_{A_1}}v_\eps(\sigma)},\]
    with quantities that go to zero as $\eps$ goes to zero.
    The main difference is, however, that the set $\Omega_{2,\eps}$ varies with $\eps$. The idea is then to map $\Omega_{2,\eps}$ to the limit set $\Omega_2$ using a diffeomorphism $\Phi$ and to study the asymptotic behaviour of the auxiliary function $\tilde{v}_\eps=v_\eps(\Phi)$.
 \end{proof}

 Analogously, we have the following theorem.

 \begin{teor}
 Let $v_\eps$ be the solution to equation \eqref{BCFint}. If $\alpha=0$, then $v_\eps-(v_\eps)_{\Omega_1}$ converges, uniformly in compact subsets of $\Omega_1$ , to the function $v_1$ solution to 
    \[\begin{cases}
    -k\divv(A_1\nabla v_1)=f_1 &\text{in }\Omega_1,\\[5pt]
    \dfrac{hk}{\gamma_{A_2}}\partial_{\nu_{A_1}}v_1=-\zeta &\text{on }\partial\Omega_1,\\[7pt]
    (v_1)_{\Omega_1}=0,
    \end{cases}
    \]
    where
    \[\zeta=\dfrac{\displaystyle\int_{\Omega_1} f_1\,dx}{\displaystyle\int_{\partial\Omega_1}\dfrac{\gamma_{A_2}}{h}\,d\sigma}.\]
    While, in compact subsets of $\Omega_2$, $v_\eps$ converges uniformly to the function $v_2$ solution to
    \[\begin{cases}
    -k\divv(A_3\nabla v_2)=f_3 &\text{in }\Omega_2,\\[5pt]
    \dfrac{hk}{\gamma_{A_2}}\partial_{\nu_{A_3}}v_2=-\zeta &\text{on }\partial\Omega_1,\\[7pt]
    v_2=0 &\text{on }\partial D.
    \end{cases}
    \]
    
 \end{teor}

\subsection{The result by Acerbi \& Buttazzo}\label{sbsAB}
In this section, we discuss the results proved by Acerbi \& Buttazzo in \cite{BCF80}. In the paper, the authors prove the $\Gamma$-convergence of variational energies related to the boundary reinforcement for the $p$-Laplace equation.\medskip

As in the previous section, let $\Omega$ to be a bounded open set in $\R^n$ with $C^{1,1}$ boundary, \[\Sigma_\eps=\Set{\sigma+t\nu(\sigma)\colon \sigma\in\partial\Omega,\,t\in(0,\eps h(\sigma))},\]
 where $h\colon\partial\Omega\to\R$ is a strictly positive Lipschitz function and denote by $D_\eps=\bar{\Omega}\cup\Sigma_\eps$. The model equation one should keep in mind is, for a given $p>1$,
 \begin{equation}\label{AB0}
     \begin{cases}
         -k\Delta_p u_\eps=f &\text{in }\Omega,\\[5pt]
         -k_\eps\Delta_p u_\eps=f &\text{in }\Sigma_\eps,\\[5pt]
         u_\eps|_\Omega=u_\eps|_{\Sigma_\eps} &\text{on }\partial\Omega,\\[5pt]
        k\abs{\nabla u_\eps}^{p-2}\partial_\nu u_\eps|_\Omega=k_\eps\abs{\nabla u_\eps}^{p-2}\partial_\nu u_\eps|_{\Sigma_\eps} &\text{on }\partial\Omega,\\[5 pt]
        u_\eps=0 &\text{on }\partial D_\eps,
     \end{cases}
 \end{equation}
 where $\Delta_pu_\eps=\divv(\abs{\nabla u_\eps}^{p-2}\nabla u_\eps)$ is the $p$-Laplace operator. Equation \eqref{AB0} is the Euler-Lagrange equation of the energy functional 
 \[\mathcal{F}_\eps(v)=k\int_\Omega\abs{\nabla v}^p\,dx+k_\eps\int_{\Sigma_\eps}\abs{\nabla v}^p\,dx-p\int_{\R^n} fv\,dx,\]
 where $v\in W^{1,p}_0(D_\eps)$. Let 
 \[\alpha=\lim_{\eps\to0}\dfrac{k_\eps}{\eps^{p-1}}.\]
 Then the simplest form of the result states that, if $f\in L^{p'}(\R^n)$, where $p'$ is the H{\"o}lder conjugate exponent of $p$, then $\mathcal{F}_\eps$ $\Gamma$-converges in the $L^p$-topology to the energy
 \[\mathcal{G}_\alpha(v)=k\int_\Omega\abs{\nabla v}^p\,dx+\alpha\int_{\partial\Omega}\dfrac{\abs{v}^p}{h^{p-1}}\,d\sigma-p\int_\Omega fv\,dx,\]
 for every $v\in W^{1,p}(\Omega)$, if $\alpha$ is finite, and to 
 \[\mathcal{G}_\infty(v)=k\int_\Omega\abs{\nabla v}^p\,dx-p\int_\Omega fv\,dx,\]
 for every $v\in W^{1,p}_0(\Omega)$ if $\alpha=+\infty$. As we will see, if $\alpha>0$, by the properties of $\Gamma$-convergence, this implies the $L^p$-convergence of $u_\eps$ to the function $u$ solution to the limit problem
 \[\begin{cases}
 -k\Delta_p u_\eps=f &\text{in }\Omega,\\[7pt]
 \dfrac{h^{p-1}k}{\alpha}\abs{\nabla u}^{p-2}\partial_\nu u+\abs{u}^{p-2}u=0 &\text{on }\partial\Omega.\medskip
 \end{cases}\]

Before we state the main result in all generality, let us recall the definition of $\Gamma$-convergence (we refer the reader to \cite{DM93} for the main properties of $\Gamma$-convergence, see also \cite{DGF75, M76, DGDM83}).
\begin{defi}
    Let $X$ be a metric space and let $F_\eps\colon X\to\bar{\R}$ be a family of functions. We say that a function $F\colon X\to\bar{\R}$ is the $\Gamma$-limit of $F_\eps$ as $\eps$ goes to zero if\begin{itemize}
        \item[i)] for every $x\in X$ and  for every family $x_\eps\in X$ such that $x_\eps$ converges to $x$, 
        \[F(x)\le \liminf_{\eps\to0}{F_\eps(x_\eps)};\]
        \item[ii)] for every $x\in X$ there exists a family $x_\eps\in X$ such that $x_\eps$ converges to $x$ and 
        \[F(x)\ge \limsup_{\eps\to0}F_\eps(x_\eps).\]
    \end{itemize}
\end{defi}

Fix $p>1$ and let $G\colon W^{1,p}(\Omega)\to[0,+\infty]$ be such that 
\begin{itemize}
    \item [1)] $G$ is lower semicontinuous with respect to the $L^{p}(\Omega)$ topology,
    \item[2)] For every $v\in W^{1,p}(\Omega)$ 
    \[G(v)\ge k\int_\Omega \abs{\nabla v}^p\,dx.\]
\end{itemize}
When $v\in L^p(\R^n)\cap W^{1,p}(\Omega)$ we will write $G(v)$ instead of $G(v|_\Omega)$.\medskip

Let $g\colon\R^n\times\R^n\to\R$ be such that\begin{itemize}
    \item[1)] for every $x\in\R^n$ $g(x,\cdot)$ is convex,
    \item[2)] there exists an increasing and continuous function $\omega\colon[0,+\infty)\to[0,+\infty)$ with $\omega(0)=0$ such that
    \[\abs{g(x_1,z)-g(x_2,z)}\le \omega(\abs{x_1-x_2})\left(1+\abs{z}^p\right),\]
    \item[3)] there exists $0<c_1<c_2$ such that for every $x,z\in\R^n$
    \[c_1\abs{z}^p\le g(x,z)\le c_2 \left(1+\abs{z}^p\right),\]
    \item[4)] there exists a continuous function $\gamma(x,z)$ which is convex and $p$-homogeneous in $z$, such that
    \[\sup\Set{\abs{g(x,z)-\gamma(x,z)}\colon x\in\R^n}\le \rho(\abs{z})\left(1+\abs{z}^p\right),\]
    where $\rho\colon[0,+\infty)\to[0,+\infty)$ is continuous, decreasing and vanishes at infinity.
\end{itemize}
For every $\eps>0$ and $v\in L^p(\R^n)$ let 
\[F_\eps(v)=\begin{cases}
    \displaystyle G(v)+k_\eps\int_{\Sigma_\eps}g(x,\nabla v)\,dx &\text{if }v\in W^{1,p}_0(D_\eps),\\[5pt]
    +\infty &\text{if }v\in L^p(\R^n)\setminus W^{1,p}_0(D_\eps).
\end{cases}\]
We define, for $\alpha\in[0,+\infty)$
\[G_\alpha(v)=\begin{cases}
    \displaystyle G(v)+\alpha\int_{\partial\Omega}\dfrac{\gamma(\sigma,\nu)}{h(\sigma)^{p-1}}\abs{v}^p\,d\sigma &\text{if }v\in W^{1,p}(\Omega),\\[5pt]
    +\infty &\text{if }v\in L^p(\R^n)\setminus W^{1,p}(\Omega),
\end{cases}\]
and
\[G_\infty(v)=\begin{cases}
    \displaystyle G(v)&\text{if }v\in W^{1,p}_0(\Omega),\\[5pt]
    +\infty &\text{if }v\in L^p(\R^n)\setminus W^{1,p}_0(\Omega).
\end{cases}\]

The main result of \cite{AB86} is the following.

\begin{teor}\label{ABTh}
    Let 
    \[\alpha=\lim_{\eps\to0}\dfrac{k_\eps}{\eps^{p-1}}\in[0,+\infty],\]
    then $F_\eps$ $\Gamma$-converges, in the strong $L^p(\R^n)$ topology, to $G_\alpha$. 
\end{teor}

\begin{oss}\label{oss+cont}
    We remark that, if $f\in L^{p'}(\R^n)$, then the function 
    \[v\in L^p(\R^n)\mapsto\int_{\R^n} fv\,dx\]
    is continuous in the strong $L^p(\R^n)$ topology. Hence by \autoref{ABTh} and well known results about $\Gamma$-convergence (see for instance \cite[Proposition 6.21]{DM93}), we have that
    \[F_\eps(v)+\int_{\R^n}fv\,dx\]
    $\Gamma$-converges, in the strong $L^p(\R^n)$ topology, to 
    \[G_\alpha(v)+\int_{\R^n}fv\,dx.\]
\end{oss}

\subsubsection*{Proof of the result}

To prove \autoref{ABTh}, we will need the following technical lemmas, whose proof we omit. 

\begin{lemma}\label{lemmaIII.1}
Let $a\colon\R^n\to\R$ be a continuous function and let $v\in W^{1,p}(\R^n)$, then
\[\lim_{\eps\to0}\dfrac{1}{\eps}\int_{\Sigma_\eps}a(x)\abs{v(x)}^p\,dx=\int_{\partial\Omega} h(\sigma) a(\sigma) \abs{v(\sigma)}^p\,d\sigma\]
\end{lemma}

\begin{lemma}\label{lemmaIII.2}
    Let $j\colon\R^n\to\R$ be a $C^1$. convex, $p$-homogeneous function such that for every $z\in\R^n$
    \[c \abs{z}^p\le j(z)\le C \abs{z}^p.\]
    Then, for every $y,z\in\R^n$ we have
    \[\abs{y\cdot\nabla j (y)}^p j(z)\ge\abs{z\cdot\nabla j (y)}^p j(y)\]
\end{lemma}

\begin{proof}[proof of \autoref{ABTh}]
An important step in the proof is to substitute, in the integrals, $g$ with the function $\gamma$. As this step is mainly technical, we will prove the theorem for the energy
\[\tilde{F}_\eps(v)=\begin{cases}
    \displaystyle G(v)+k_\eps\int_{\Sigma_\eps}\gamma(x,\nabla v)\,dx &\text{if }v\in W^{1,p}_0(D_\eps),\\[5pt]
    +\infty &\text{if }v\in L^p(\R^n)\setminus W^{1,p}_0(D_\eps).
\end{cases}\]
In addition, we will also assume $\gamma$ to be $C^1$, and, by the assumptions on $g$, we have that
\begin{itemize}
    \item[1)] for every $x_1,x_2,z\in\R^n$
    \[\abs{\gamma(x_1,z)-\gamma(x_2,z)}\le \omega(\abs{x_1-x_2})\abs{z}^p,\]
    \item[2)]  $x,z\in\R^n$
    \[c_1\abs{z}^p\le \gamma(x,z)\le c_2 \abs{z}^p.\]
\end{itemize}\medskip

We start by proving the limsup inequality. 
If $\alpha=+\infty$, for every $\eps>0$, $G_\infty\ge \tilde{F}_\eps$ pointwise so that for every $v\in L^p(\R^n)$ setting $v_\eps=v$ we trivially have
\[G_\infty(v)\ge\limsup_{\eps\to0}\tilde{F}_\eps(v).\]

Assume now $\alpha<+\infty$, and let $v\in L^p(\R^n)$. If $v\notin W^{1,p}(\Omega)$, then $G_\alpha(v)=+\infty$ and the inequality is trivial, hence let $v\in W^{1,p}(\Omega)$. By the regularity of $\Omega$ we can assume $v\in W^{1,p}(\R^n)$. recall that, when $\eps$ is sufficiently small, every $x\in\Sigma_\eps$ can be uniquely written as $x=\sigma(x)+d(x)\nu(\sigma(x))$, where $d(x)$ is the distance from $\Omega$ and $\sigma(x)$ is the orthogonal projection of $\partial\Omega$, let $h(x)=h(\sigma(x))$ and
\[\phi_\eps(x)=\begin{cases}
    1 &\text{in }\Omega,\\[5pt]
    1-\dfrac{d(x)}{\eps h(x)} &\text{if }x\in\Sigma_\eps,\\[5pt]
    0 &\text{if }x\in \R^n\setminus D_\eps.
\end{cases}\]
Then $0\le\phi_\eps\le1$ and the function $v_\eps=v\phi_\eps\in W^{1,p}_0(D_\eps)$, and $v_\eps$ converges to $v\chi_\Omega$ in $L^p(\Omega)$. For every $t\in(0,1)$ we have
\begin{equation}\label{sup1}\begin{split}\tilde{F}_\eps(v_\eps)&=G(v)+k_\eps\int_{\Sigma_\eps}\gamma(x,v\nabla\phi_\eps+\phi_\eps\nabla v)\,dx\\[15pt]
&\le G(v)+k_\eps\int_{\Sigma_\eps}\left[t\abs{v}^p\gamma\left(x,\dfrac{\nabla\phi_\eps}{t}\right)+(1-t)\phi_\eps^p\gamma\left(x,\dfrac{\nabla v}{1-t}\right)\right]\,dx\\[15pt]
&\le G(v)+k_\eps t^{1-p}\int_{\Sigma_\eps}\abs{v}^p\gamma(x,\nabla\phi_\eps)\,dx+c_2 k_\eps (1-t)^{1-p}\int_{\Sigma_\eps}\abs{\nabla v}^p\,dx.
\end{split}\end{equation}
By direct computations, in $\Sigma_\eps$ we have
\[\nabla\phi_\eps=-\dfrac{1}{\eps h}\nu+\dfrac{d}{\eps h^2}\nabla h=d\dfrac{1}{\eps h^2}\nabla h-(1-d)\dfrac{1}{(1-d)\eps h}\nu,\]
so that, by convexity and $ p$-homogeneity 
\begin{equation}\label{sup2}\begin{split}k_\eps\int_{\Sigma_\eps}\abs{v}^p\gamma(x,\nabla \phi_\eps)\,dx&\le k_\eps\int_{\Sigma_\eps}\abs{v}^p\left[\dfrac{d}{(\eps h^2)^p}\gamma(x,\nabla h)+\dfrac{(1-d)^{1-p}}{(\eps h)^p}\gamma(x,\nu)\right]\,dx\\[15 pt]\eps&\le \dfrac{k_\eps}{\eps^{p-1}}c_2\int_{\Sigma_\eps}\abs{v}^ph^{1-2p}\abs{\nabla h}^p\,dx+\dfrac{k_\eps}{\eps^{p-1}}\dfrac{1}{\eps}\int_{\Sigma_\eps}\dfrac{\abs{v}^p}{h^p}\gamma(x,\nu)\,dx.\end{split}\end{equation}
Joining \eqref{sup1} and \eqref{sup2} and passing to the limsup for $\eps$ going to zero, we have that for every $t\in(0,1)$
\begin{equation}\label{sup3}
    \limsup_{\eps\to0}\tilde{F}_\eps(v_\eps)\ge G(v)+\alpha t^{1-p}\limsup_{\eps\to0}\dfrac{1}{\eps}\int_{\Sigma_\eps}\dfrac{\abs{v}^p}{h^p}\gamma(x,\nu)\,dx,
\end{equation}
Finally, using \autoref{lemmaIII.1}, and passing to the limit for $t$ going to $1$, we have 
\[\limsup_{\eps\to0}\tilde{F}_\eps(v_\eps)\ge G(v)+\alpha\int_{\partial\Omega}\dfrac{\abs{v}^p}{h^{p-1}}\gamma(\sigma,\nu)\,d\sigma=G_\alpha(v).\]\medskip

We now prove the liminf inequality. Let $v\in L^p(\R^n)$ and $v_\eps\in L^p(\R^n)$ with $v_\eps$ converging to $v$ in $L^p(\R^n)$. If 
\[\liminf_{\eps\to0}\tilde{F}_\eps(v_\eps)=+\infty,\]
the inequality is trivial, hence we can assume $\tilde{F}_\eps(v_\eps)\le C$ for every $\eps$ so that $v_\eps\in W^{1,p}_0(D_\eps)$ for every $\eps$. Moreover, 
\[k\int_{\Omega} \abs{\nabla v_\eps}^p\,dx+k_\eps c_1\int_{\Sigma_\eps}\abs{\nabla v_\eps}^p\,dx\le \tilde{F}_\eps(v_\eps)\le C,\]
whence $v\in W^{1,p}(\Omega)$ and, up to a subsequence (not relabelled), $v_\eps$ converges to $v$ weakly in $W^{1,p}(\Omega)$.\medskip

If $\alpha<+\infty$, by semicontinuity, 
\[\liminf_{\eps\to0}G(v_\eps)\le G(v),\]
and we are left to prove that
\begin{equation}\label{inf0}\liminf_{\eps\to0}k_\eps\int_{\Sigma_\eps} \gamma(x,\nabla v_\eps)\,dx\le \alpha\int_{\partial\Omega}\dfrac{\gamma(\sigma,\nu)}{h(\sigma)^{p-1}}\abs{v}^p\,d\sigma.\end{equation}

For every $\sigma\in\partial\Omega$ let 
\[\nu_\gamma(\sigma)=\dfrac{1}{p}\nabla_z\gamma(\sigma,\nu(\sigma)).\]
Then, by Euler's theorem on homogeneous functions, we have that
\[\nu_\gamma\cdot\nu=\gamma(\sigma,\nu)\ge c_1>0,\]
so that, arguing as in \autoref{conormalgraph}, we can describe $\Sigma_\eps$ as
\[\Sigma_\eps=\Set{\sigma+t\nu_\gamma(\sigma)\colon \sigma\in\partial\Omega,\,t\in(0, h_\eps(\sigma))},\]
for some positive function $h_\eps$, with 
\begin{equation}\label{gammagamma}\lim_{\eps\to0}\dfrac{\eps h(\sigma)}{h_\eps (\sigma)}=\gamma(\sigma,\nu),\end{equation}
uniformly on $\partial\Omega$. By the change of variable $x=\sigma+t\nu_\gamma$, we have 
\[\int_{\Sigma_\eps} \gamma(x,\nabla v_\eps)\,dx=\int_{\partial\Omega}\int_0^{h_\eps(\sigma)} \gamma(\sigma+t\nu_\gamma,\nabla v_\eps)\gamma(\sigma,\nu)(1+tq(t,\sigma))\,dt\,d\sigma,\]
where $q$ is bounded. Notice that, 
\[\abs{\gamma(\sigma+t\nu_\gamma,\nabla v_\eps)-\gamma(\sigma,\nabla v_\eps)}\le \omega(h_\eps)\abs{\nabla v_\eps}^p,\]
so that 
\[k_\eps\int_{\Sigma_\eps} \gamma(x,\nabla v_\eps)\,dx=k_\eps\int_{\partial\Omega}\int_0^{h_\eps(\sigma)} \gamma(\sigma,\nabla v_\eps)\gamma(\sigma,\nu)\,dt\,d\sigma+r_\eps,\]
with $r_\eps$ going to zero with $\eps$. By \autoref{lemmaIII.2} we have 
\[\gamma(\sigma,\nabla v_\eps)\ge \dfrac{\abs{\nabla v_\eps\cdot\nabla_z\gamma(\sigma,\nu) }^p}{\abs{\nu\cdot\nabla_z\gamma(\sigma,\nu)}^p}\gamma(\sigma,\nu)=\abs{\nabla v_\eps\cdot\nu_\gamma }^p\gamma(\sigma,\nu)^{1-p},\]
hence
\begin{equation}\label{inf1}k_\eps\int_{\Sigma_\eps} \gamma(x,\nabla v_\eps)\,dx\ge k_\eps\int_{\partial\Omega}\int_0^{h_\eps(\sigma)} \abs{\nabla v_\eps\cdot\nu_\gamma }^p\gamma(\sigma,\nu)^{2-p}\,dt\,d\sigma+r_\eps.\end{equation}
On the other hand, for $\Hn$-almost every $\sigma\in\partial\Omega$
\begin{equation}\label{tfci}\abs{v_\eps(\sigma)}^p=\abs*{\int_0^{h_\eps(\sigma)}\nabla v_\eps(\sigma+t\nu_\gamma)\cdot\nu_\gamma\,dt}^p\le \abs{ h_\eps(\sigma)}^{p-1}\int_0^{h_\eps(\sigma)}\abs{\nabla v_\eps(\sigma+t\nu_\gamma)\cdot\nu_\gamma}^p\,dt.\end{equation}
Using \eqref{inf1} and \eqref{tfci} we get
\begin{equation}\label{inf2}\displaystyle k_\eps\int_{\Sigma_\eps} \gamma(x,\nabla v_\eps)\,dx\ge \dfrac{k_\eps}{\eps^{p-1}}\int_{\partial\Omega} \abs*{\dfrac{h_\eps(\sigma) \gamma(\sigma,\nu)}{\eps h}}^{1-p}\dfrac{\gamma(\sigma,\nu)}{h(\sigma)^{p-1}}\abs{v_\eps }^p\,d\sigma+r_\eps.\end{equation}
Finally, using \eqref{gammagamma} and passing to the limit in \eqref{inf2}, by the convergence of $v_\eps$ to $v$ in $L^p(\partial\Omega)$, we have \eqref{inf0}.\medskip

If $\alpha=+\infty$, we have that for every $\alpha_0>0$ and $\eps>0$ sufficiently small, $k_\eps>\alpha_0\eps^{p-1}$, hence 
\[\tilde{F}_\eps(v_\eps)\ge G(v_\eps)+\alpha_0\eps^{p-1}\int_{\Sigma_\eps}\gamma(x,v_\eps)\,dx.\]
Using the liminf inequality proved for the case $\alpha<+\infty$ and passing to the limit for $\eps$ going to zero, we have
\[\liminf_{\eps\to0}\tilde{F}_\eps(v_\eps)\ge G(v)+\alpha_0\int_{\partial\Omega}\dfrac{\gamma(\sigma,\nu)}{h(\sigma)^{p-1}}\abs{v}^p\,d\sigma.\]
Then, passing to the limit for $\alpha_0$ going to infinity, we have 
\[\liminf_{\eps\to0}\tilde{F}_\eps(v_\eps)\ge\begin{cases}
    G(v)
&\text{if } v\in W^{1,p}_0(\Omega),\\[5pt]
+\infty &\text{if } v\in L^p(\R^n)\setminus W^{1,p}_0(\Omega),
\end{cases}\] that is, 
\[\liminf_{\eps\to0}\tilde{F}_\eps(v_\eps)\ge G_\infty(v).\]
\end{proof}

In \cite{BMM97}, the authors generalised the result to the case of \emph{non homogeneous} reinforcement, that is, they considered the case of the $p$-Laplacian equation with different exponent in $\Omega$ and $\Sigma_\eps$. Moreover, they allow $k_\eps$ to be a function with possibly different asymptotic behaviours with respect to $\eps^{p-1}$ on different subsets of $\partial\Omega$. Another direct generalisation of the result is the one contained in \cite{MMP08}, where the authors study the case of a right-hand side in $L^1$.

\subsubsection*{Convergence of minima}

It is well known that the $\Gamma$-convergence of a sequence of functionals is strictly related to the convergence of its minimum points and values, for instance, we have the following proposition (we refer the reader to \cite[Theorem 7.8]{DM93}).

\begin{prop}
    Let $X$ be a metric space and let $F,F_\eps\colon X\to\bar{\R}$ functions such that
    \begin{itemize}
        \item[i)] $F_\eps$ is equi-coercive, that is, for every $c>0$ the exists a compact subset $K_c\subset X$ such that
        \[\Set{x\in X|\,F_\eps(x)\le c}\subseteq K_c,\]
        for every $\eps>0$;
        \item[ii)] $F_\eps$ $\Gamma$-converges to $F$.
    \end{itemize}
    Then, $F$ admits minimum in $X$ and
    \[\min_X F=\lim_{\eps\to0}\left(\inf_X F_\eps\right).\]
    Moreover, if $x_\eps$ if a family in $X$ such that
    \[\lim_{\eps\to0}F_\eps(x_\eps)=\lim_{\eps\to0}\left(\inf_X F_\eps\right)\]
    and $x_\eps$ converges to $x_0$, then \[F(x_0)=\min_XF.\]
\end{prop}

Let $f\in L^{p'}(\R^n)$, we want to apply the previous proposition to prove that, when $\alpha$ is strictly positive, the minimisers of the functional
\[\mathcal{F}_\eps(v)=F_\eps(v)+\int_{\R^n} fv\,dx\]
converge to the ones of
\[\mathcal{G}_\alpha(v)=G_\alpha(v)+\int_{\R^n} fv\,dx.\]
By \autoref{oss+cont}, $\mathcal{F}_\eps$ $\Gamma$-converges to $\mathcal{G}_\alpha$ so that we just need to show the equi-coercivity of $\mathcal{F}_\eps$. 

\begin{prop}
    Let $\alpha>0$. Let $v_\eps$ be a family in $L^p(\R^n)$ such that 
    \[\mathcal{F}_\eps(v_\eps)\le c\]
    for some $c>0$, then $v_\eps$ is relatively compact in $L^p(\R^n)$.
\end{prop}
\begin{proof}

    By our assumptions on $F_\eps$, we have
    \[k\int_\Omega \abs{\nabla v_\eps}^p\,dx+k_\eps c_1 \int_{\Sigma_\eps}\abs{\nabla v_\eps}^p\,dx+\int_{D_\eps}fv_\eps\,dx\le \mathcal{F}_\eps(v_\eps)\le c.\]
    Then, using H\"{o}lder's Young's and  inequalities, we have that for every $\delta>0$
    \begin{equation}\label{int0}\int_\Omega \abs{\nabla v_\eps}^p\,dx+k_\eps  \int_{\Sigma_\eps}\abs{\nabla v_\eps}^p\,dx\le c_{\delta,f} +\delta \int_{D_\eps} \abs{v_\eps}^p\,dx.\end{equation}

    For every $t\in[0,\eps h]$ and $\sigma\in\partial\Omega$ we have 
    \[\abs{v_\eps(\sigma+t\nu)}^p=\abs*{\int_t^{\eps h(\sigma)}\partial_\nu v_\eps(\sigma+s\nu)\,ds}^p\le  \left(\eps h(\sigma)\right)^{p-1}\int_0^{\eps h(\sigma)}\abs{\nabla v_\eps(\sigma+s\nu)}^p\,ds,\]
    so that, integrating, we get
    \begin{equation}\label{int1}
    \int_{\Sigma_\eps} \abs{v_\eps}^p\,dx\le C \eps^{p}\int_{\Sigma_\eps}\abs{\nabla v_\eps}^p\,dx \le C \eps k_\eps \int_{\Sigma_\eps}\abs{\nabla v_\eps}^p\,dx
    \end{equation}
    and
     \begin{equation}\label{int2}
     \int_{\partial\Omega} \abs{v_\eps}^p\,d\sigma\le C \eps^{p-1}\int_{\Sigma_\eps}\abs{\nabla v_\eps}^p\,dx\le C k_\eps \int_{\Sigma_\eps}\abs{\nabla v_\eps}^p\,dx,
    \end{equation}
    where we used the fact that, by the assumption on $\alpha>0$, we have that $\eps^{p-1}/k_\eps $ is bounded. Using a Poincaré-type inequality with trace term (also known as Friedrichs' inequality) and \eqref{int2}, we have 
    \begin{equation}\label{int3}\int_\Omega \abs{v_\eps}^p\,dx\le C\left[\int_\Omega \abs{\nabla v_\eps}^p\,dx+\int_{\partial\Omega} \abs{v_\eps}^p\,d\sigma\right]\le C\left[\int_\Omega \abs{\nabla v_\eps}^p\,dx+ k_\eps \int_{\Sigma_\eps}\abs{\nabla v_\eps}^p\,dx\right].\end{equation}
    Then, using \eqref{int1} and \eqref{int3} in \eqref{int0}, for a suitable choice of $\delta>0$ we finally have
    \[\int_\Omega \abs{\nabla v_\eps}^p\,dx+k_\eps  \int_{\Sigma_\eps}\abs{\nabla v_\eps}^p\,dx\le C.\]
    Hence $v_\eps$ in uniformly bounded in $W^{1,p}(\Omega)$ and, by \eqref{int1},
    \[\lim_{\eps\to0}\int_{\Sigma_\eps} \abs{v_\eps}^p\,dx=0,\]
    which proves the assertion.
\end{proof}

\subsection{Subsequent developments and other boundary conditions}\label{sbsnew}

The basic result of the previous section can be summarised as follows. Given a smooth $C^{1,1}$ bounded open set $\Omega\subset\R^n$, and thin layer $\Sigma_\eps$ described by normal graph on $\partial\Omega$ of the function $\eps h$, where $h\colon \partial\Omega\to(0,+\infty)$ is a Lipschitz function, for any $f\in L^2(\R^n)$, the solution $u_\eps$ to the diffraction problem
\begin{equation}\label{eqbasic}\begin{cases}
    -\divv(A_1\nabla u_\eps) = f &\text{in }\Omega,\\[5pt]
    -\eps\divv(A_2\nabla u_\eps) = f &\text{in }\Sigma_\eps,\\[5pt]
    u_\eps|_\Omega=u_\eps|_{\Sigma_\eps} &\text{on }\partial\Omega,\\[5pt]
    \partial_{\nu_{A_1}} u_\eps|_\Omega=\eps\partial_{\nu_{A_2}} u_\eps|_{\Sigma_\eps} &\text{on }\partial\Omega,\\[5pt]
    u_\eps=0 &\text{on }\partial D_\eps,
\end{cases}\end{equation}
with Dirichlet boundary condition on the boundary of $D_\eps=\bar{\Omega}\cup\Sigma_\eps$, converges to the solution to the Robin boundary value problem
\begin{equation}\label{eqlimbasic}\begin{cases}
    -\divv(A_1 \nabla u) = f &\text{in }\Omega,\\[5 pt]
    \partial_{\nu_{A_1}} u + \dfrac{\gamma_{A_2}}{h}u=0 &\text{on }\partial\Omega,
\end{cases}\end{equation}

where $\gamma_{A_2}=\nu_{A_2}\cdot\nu$ and $\nu_{A_1},\nu_{A_2}$ are the conormal directions $A_1 \nu$,$A_2\nu$ on $\partial\Omega$ from "inside" and "outside" $\Omega$ respectively.\medskip

In many applications, however, it could be useful to work with less regular sets, for instance, let $\Omega$ be a bounded open set with Lipschitz boundary and  let $D_\eps$ be a generic family of bounded open sets containing $\Omega$ such that
\[d_\eps=\sup\Set{d(x,\Omega)\colon x\in D_\eps},\]
goes to zero as $\eps$ goes to zero. Under these mild assumptions, in \cite{BDMM89}, the authors derive the necessary and sufficient condition to have that the solution to the diffraction problem \eqref{eqbasic}, for any $f$, converges, in $L^2$, to the solution to a boundary value problem of the type \eqref{eqlimbasic}. The boundary parameter $\gamma_{A_2}/h$, however, must be replaced with a (capacitary) measure $\mu_A$, which is supported by $\partial\Omega$ and depends uniquely on the family $D_\eps$ (and on the elliptic operator). The limit problem can be formally written as

\[\begin{cases}
    -\divv(A_1\nabla u) = f &\text{in }\Omega,\\[5pt]
    \partial_{\nu_{A_1}} u +\mu_A u=0 &\text{on }\partial\Omega,
\end{cases}\]
and has to be interpreted, variationally, as the first order optimality condition for the minimisation in $H^1(\Omega)$ of the energy
\[\int_\Omega A_1\nabla v\cdot\nabla v\,dx+\int v^2\,d\mu_A-2\int_\Omega fv\,dx.\]
As the necessary and sufficient condition is quite involved, we refer the reader to \cite[Theorem 6.1]{BDMM89} for the precise statement. We remark, however, that the simple condition \[d_\eps<c\eps,\] for some $c>0$, is a sufficient condition for the limit to hold. We also refer the reader, for instance, to \cite{CV11, CV15} and the references therein, for the reinforcement of fractal sets, to \cite{AKK26, AKKtris} for the \emph{partial} reinforcement of a piece-wise flat subset of the boundary, and to \cite{EJEAEMERS24} for the case of a vectorial problem.
\medskip

Other generalisations include, for instance, the case of the reinforcement with a thin layer with \emph{oscillating thickness} considered by \cite{BK87} (see also \cite{MV02, LZ11} for more recent results on the topic); the reinforcement for the \emph{torsion problem in multiconnected domains} studied in \cite{BGMM02} and \cite{MV02}, were the function $u_\eps$ is required to have a constant, unprescribed, trace on the inner boundaries and a prescribed flux condition. The case of different boundary conditions on $\partial D_\eps$ as also been studied. \medskip

In particular, in \cite{MS03}, the authors consider the case of \emph{Neumann boundary conditions}. Namely, let $\Omega\subset\R^n$ be a smooth bounded set, $\Sigma_\eps$ an uniform layer of thickness $\eps>0$, $D_\eps=\bar{\Omega}\cup\Sigma_\eps,$ and consider the boundary value problem
\begin{equation}\label{eqN}\begin{cases}
        -k\Delta u_{\eps} = f &\text{in } \Omega_\eps\\[5 pt]
         -k_\eps\Delta u_{\eps} = f &\text{in } \Sigma_\eps\\[5 pt]
         u_\eps|_{\Omega_\eps}=u_\eps|_{\Sigma_\eps} &\text{on } \partial\Omega,\\[5 pt]
         k \partial_\nu u_\eps|_{\Omega_\eps} = k_\eps \partial_\nu u_\eps|_{\Sigma_\eps} &\text{on }\partial\Omega,\\[5 pt]
         \partial_{\nu_\eps}u_\eps=0 &\text{ on } \partial D_\eps,\end{cases}\end{equation}
         where $f\in L^2(\R^n)$ satisfies the compatibility condition 
         \[\int_\Omega f\,dx=\int_{\Sigma_\eps}f\,dx=0,\]
         for every $\eps>0$. The authors prove that, if $k_\eps$ goes to zero as $\eps$ goes to zero, and $u_\eps$ is a solution to \eqref{eqN}, denoting by $(u_\eps)_\Omega$ its mean in $\Omega$, the function $u_\eps-(u_\eps)_\Omega$ converges, uniformly in the compact subset of $\Omega$, to the solution to
         \[\begin{cases}
              -k\Delta v=f &\text{in }\Omega,\\[5 pt]
         \partial_\nu v =0 &\text{on }\partial\Omega,\\[5 pt]
         \displaystyle\int_\Omega v\,dx=0.
         \end{cases}\]
          It is also possible to consider the case in which $k_\eps$ goes to infinity as $\eps$ goes to zero. The limit problem will then depend on the limit $\eta$ of $\eps k_\eps$. In particular, letting $v_\eps$  be the solution to \eqref{eqN} with zero mean in $D_\eps$, if $\eta$ is finite, then $v_\eps$ converges, strongly in $H^1(\Omega)$ to the solution of the following Wentzel problem
         \[\begin{cases}
             -k\Delta v = f &\text{in }\Omega,\\[5 pt]
             k\partial_\nu v + \eta \Delta_{\partial\Omega} v =0 &\text{on }\partial\Omega,\\[5 pt]
             \displaystyle\int_\Omega v\,dx=0,
         \end{cases}
         \]
         where $\Delta_{\partial\Omega}$ is the tangential Laplacian on $\partial\Omega$. If $\eta=+\infty$, $v_\eps$ converges to $w-(w)_\Omega$, where $w$ is the solution to 
         \[\begin{cases}
             -k\Delta w = f &\text{in }\Omega,\\[5 pt]
             w=0 &\text{on }\partial\Omega.
         \end{cases}\]
We also remark that in the same paper, the authors also consider mixed Dirichlet-Neumann boundary conditions for particular geometries, and use $\Gamma$-convergence to study semi-linear equations with Dirichlet boundary conditions.\medskip

The case of \emph{Robin boundary conditions}, on the other hand, has been more recently considered by \cite{LWZZ12} (see also \cite{DPNST21, ACNT24, AC25}). Namely, consider the boundary value problem
\begin{equation}\label{eqR}\begin{cases}
        -k\Delta u_{\eps} = f &\text{in } \Omega_\eps\\[5 pt]
         -k_\eps\divv(A_2 \nabla u_{\eps}) = f &\text{in } \Sigma_\eps\\[5 pt]
         u_\eps|_{\Omega_\eps}=u_\eps|_{\Sigma_\eps} &\text{on } \partial\Omega,\\[5 pt]
         k \partial_\nu u_\eps|_{\Omega_\eps} = k_\eps \partial_{\nu_{A_2}} u_\eps|_{\Sigma_\eps} &\text{on }\partial\Omega,\\[5 pt]
         k_\eps\partial_{\nu_{\eps,A_2}}u_\eps+\beta u_\eps=0 &\text{ on } \partial D_\eps,\end{cases}\end{equation}
         where $\nu_{\eps,A_2}=A_2\nu_\eps$ is the conormal to $\partial D_\eps$, and $\beta$ is a positive constant. The authors prove that, if $k_\eps$ goes to zero as $\eps$ goes to zero, and we denote by $\alpha$ the limit of $k_\eps/\eps$, the function $u_\eps$, solution to \eqref{eqR}, converges (strongly in $L^2(\Omega)$) to the solution of
         \[\begin{cases}
              -k\Delta u=f &\text{in }\Omega,\\[5 pt]
         k\partial_\nu u+\dfrac{\beta \alpha \gamma_{A_2}}{\beta+\alpha\gamma_{A_2}}u =0 &\text{on }\partial\Omega.\\[5 pt]
         \end{cases}\]
         Actually, the problem was originally studied for the parabolic case, but the same result holds for the elliptic case.
         We notice that this type of effective boundary condition was already mentioned in the book by Carslaw and Jaeger (\cite{CJ59}) in the discussion about the case of a
"\emph{thin surface skin of poor conductor}". In \cite{LWZZ12}, the authors prove the result also in the case of "\emph{optimally
aligned coating}" introduced in \cite{RW06} (see also \cite{LRWZ09} for the case of Dirichlet boundary conditions). That is, the result still holds true under milder assumptions. Namely, instead of assuming that the matrix $A_{2,\eps}=k_\eps A_2$ is small in \emph{all} directions, it is enough to assume that $A_{2,\eps}$ is bounded and small in the normal direction to $\partial\Omega$. That is, $\nu$ is an eigenvector of eigenvalue $A_{2,\eps}(x)\nu\cdot\nu=k_\eps \gamma_{A_2}(\sigma(x))$, where $\sigma(x)$ denotes the orthogonal projection on $\partial\Omega$. We also refer the reader to \cite{CPW12} (see also \cite{LLW21}) for the $2$-dimensional case in which the eigenvalue of $A_{2,\eps}$ in the tangential direction is allowed to diverge.\medskip

The corresponding eigenvalue problems have also been extensively studied, see for instance \cite{RW06, LL17, Y18, GNPM21} for the Dirichlet boundary condition, \cite{GLNP06, GLNP06bis, GMP11} for the Neumann boundary condition, and \cite{ZRWZ09, DPNST21, CV26} for the Robin one, and the references therein.\medskip 

Finally, we want to stress that the techniques of \cite{BCF80, AB86} have also been used to study reinforcement problems for different differential operators such as the clamped plate \cite{AB86bis}, diffusion equation with high specific heat \cite{CR90}, the Stefan problem for large heat conductivity \cite{SV97}, the stationary fluid flow problem \cite{EJB07,BEJ13}, and the porous medium equation \cite{CDP24}.   

\section{The optimisation point of view}\label{sopt}

In \cite{B88}, Buttazzo proposed the study of optimisation problems related to the limit equation of the reinforcement problem with Dirichlet boundary conditions
\begin{equation}\label{mod}
    \begin{cases}
        -\Delta u_h = f &\text{in }\Omega,\\[5pt]
        \partial_\nu u_h+\dfrac{1}{h}u_h=0 &\text{on }\partial\Omega.
    \end{cases}
\end{equation}
As discussed in the previous section, from a physical perspective, the solution $u_h$ to equation \eqref{mod} is an approximation of the steady-state temperature in a conductor $\Omega$, surrounded by a thin layer of insulating material. The function $f$ represents the \emph{heat source} in $\Omega$, while the function $h$ represents the \emph{shape} of the insulating layer. The optimisation problem proposed in \cite{B88} can then be stated as 

\begin{quote}
    Given $\Omega$ and $f$, what is the "best" distribution of insulator around $\partial\Omega$?
\end{quote}

To give a mathematical formalisation of the problem, an optimality criterion must be defined. A natural choice is the following. For every $h\colon\partial\Omega\to\R^+$ and $v\in H^1(\Omega)$ consider
\[\mathcal{E}(v,h)=\int_\Omega \abs{\nabla v}^2\,dx+\int_{\partial\Omega}\dfrac{v^2}{h}\,d\Hn-2\int_\Omega fv\,dx,\]
the variational energy associated to \eqref{mod}, then
\[\mathcal{E}(h)=\min_{v\in H^1(\Omega)} \mathcal{E}(v,h)=\mathcal{E}(u_h,h)=-\int_\Omega fu_h\,dx.\]
We will say that a configuration $h_1$ is \emph{better} than a configuration $h_2$ if
\[\mathcal{E}(h_1)\le\mathcal{E}(h_2).\]
Notice that, when the heat source is \emph{uniformly distributed} (i.e. $f$ is constant), the energy $\mathcal{E}(h)$ is proportional to the opposite of the average temperature in $\Omega$, hence a better configuration of insulating material corresponds to a higher average temperature. \medskip

The optimisation problem consists in minimising the energy $\mathcal{E}$ among non-negative functions $h\colon\partial\Omega\to\R^+$ which satisfy the integral constraint
\begin{equation}\label{intcost}\int_{\partial\Omega}h\,d\Hn=m.\end{equation}
The integral constraint is to be interpreted as an approximated measure constraint on the insulating layer. Indeed, if we define 
\[\Sigma_{\eps h}=\Set{\sigma+t\nu\colon \sigma\in\partial\Omega,\,t\in(0,\eps h)},\]
then we have that
\[\abs{\Sigma_{\eps h}} = \eps\int_{\partial\Omega}h\,d\Hn+O(\eps^2).\]
We refer the reader to \cite{AKK25bis} for numerical results on the problem.\medskip

In \cite{B88}, the author also proposed the investigation of other optimality criteria such as 
\[\int_\Omega \abs{u_h-u_{\text{obj}}}^2\,dx,\]
where $u_{\text{obj}}$ is a desired temperature, or more generally 
\[\int_\Omega g(x,u_h)\,dx+\int_{\partial\Omega}j(\sigma,h,u_h)\,d\Hn(\sigma),\]
for suitable choices of functions $g$ and $j$. However, to our knowledge, this type of optimisation problem is still open for the limit equation in the Dirichlet case and is partially solved in the Robin one (see \cite{CV26}) under the weaker constraint 
\[\int_{\partial\Omega} h\,d\Hn\le m.\]\medskip

The corresponding optimisation problem for the eigenvalue, that is to minimise the principal eigenvalue, $\lambda(h)$, of the problem
\begin{equation}
    \label{modeig}
    \begin{cases}
        -\Delta u = \lambda u &\text{in }\Omega,\\[5pt]
        \partial_\nu u+\dfrac{1}{h}u=0 &\text{on }\partial\Omega,
    \end{cases}
\end{equation}
under the integral constraint \eqref{intcost}, was originally proposed by Friedman in \cite{F80}. We notice that from a thermal insulation point of view, the eigenvalue $\lambda(h)$ represents the rate of decay of the temperature. We refer the reader to \cite{BB19, BBK25} for numerical results on the problem.

\subsection{Optimisation of variational energy}
Fix $m>0$, $\Omega\subset\R^n$ a bounded Lipschitz open set, and let
\[\mathcal{H}_m(\partial\Omega)=\Set{ h\in L^1(\partial\Omega)\colon\,h\ge0,\,\int_{\partial\Omega} h\,d\sigma=m
}.\]
We are interested in the minimisation problem
\begin{equation}\label{minE}
    \min\Set{\mathcal{E}(v,h)\colon v\in H^1(\Omega),\, h\in \mathcal{H}_m(\partial\Omega)}
\end{equation}
studied in \cite{B88} (see also \cite{BBN17, BBN17bis})

\begin{prop}\label{minaux1}
    For every $v\in L^2(\partial\Omega)$ there exists $h_v\in \mathcal{H}_m(\partial\Omega)$ solution to 
    \begin{equation}\label{aux1}\min_{h\in\mathcal{H}_m(\partial\Omega)}\int_{\partial\Omega}\dfrac{v^2}{h}\,d\sigma.\end{equation}
    Moreover, if $v$ is not identically zero, the solution is unique and
    \[h_v=m\dfrac{\abs{v}}{\displaystyle\int_{\partial\Omega}\abs{v}\,d\sigma}.\]
\end{prop}
\begin{proof}
    If $v$ is identically zero, every function $h$ is a solution to \eqref{aux1}. Assume $v\ne0$, by H\"{o}lder's inequality we have that for every $h\in\mathcal{H}_m(\partial\Omega)$
    \[\left(\int_{\partial\Omega} \abs{v}\right)^2\le\int_{\partial\Omega}\dfrac{v^2}{h}\,d\sigma\int_{\partial\Omega}h\,d\sigma=m\int_{\partial\Omega}\dfrac{v^2}{h}\,d\sigma,\]
    that is
    \[\int_{\partial\Omega}\dfrac{v^2}{h}\,d\sigma\ge \dfrac{1}{m}\left(\int_{\partial\Omega} \abs{v}\right)^2\]
    for every $h\in\mathcal{H}_m(\partial\Omega)$. On the other hand, setting 
     \[h_v=m\dfrac{\abs{v}}{\displaystyle\int_{\partial\Omega}\abs{v}\,d\sigma},\]
     by direct computation, we have
     \[\int_{\partial\Omega}\dfrac{v^2}{h_v}\,d\sigma= \dfrac{1}{m}\left(\int_{\partial\Omega} \abs{v}\right)^2.\]
     Finally, the uniqueness comes from the equality case in H\"{o}lder's inequality.
\end{proof}
\begin{oss}
    We remark that the \autoref{minaux1} still holds if we consider $h$ in the larger space
    \[\mathcal{H}_m^-(\partial\Omega)=\Set{ h\in L^1(\partial\Omega)\colon\,h\ge0,\,\int_{\partial\Omega} h\,d\sigma\le m
}\]
\end{oss}

Using \autoref{minaux1} we have that problem \eqref{minE} is equivalent to the auxiliary problem
\begin{equation}\label{aux2}
    \min \Set{\int_\Omega \abs{\nabla v}^2\,dx+\dfrac{1}{m}\left(\int_{\partial\Omega} \abs{v}\right)^2-2\int_\Omega fv\,dx\colon\,v\in H^1(\Omega)}.
\end{equation}
The following Poincaré-type inequality holds

\begin{prop}
    There exists $C>0$ such that, for every $v\in H^1(\Omega)$
    \begin{equation}\label{pb}
        \int_\Omega v^2\,dx\le C\left[\int_\Omega \abs{\nabla v}^2\,dx+\left(\int_{\partial\Omega} \abs{v}\right)^2\right].
    \end{equation}
\end{prop}
\begin{proof}
    Assume by contradiction that \eqref{pb} fails. Then, there exists a sequence $v_k\in H^1(\Omega)$ with 
    \begin{equation}\label{l2=1}\int_\Omega v_k^2\,dx=1\end{equation}
    and
    \[\lim_{k\to\infty}\left[\int_\Omega \abs{\nabla v_k}^2\,dx+\left(\int_{\partial\Omega} \abs{v_k}\right)^2\right]=0.\]
    Hence, up to a subsequence, $v_k$ converges to zero in $H^1(\Omega)$, which is in contradiction with \eqref{l2=1}.
\end{proof}

\begin{teor}
    Fix  $m>0$ and $f\in L^2(\Omega)$. Then, problem \eqref{minE} admits at least a solution $(u_m, h_m) \in  H^1(\Omega)\times\mathcal{H}_m(\partial\Omega)$, where $u_m$ is a minimizer to \eqref{aux2}, and, if $u_m|_{\partial\Omega}\not \equiv0$, 
    \[h_m=m\dfrac{\abs{u_m}}{\displaystyle\int_{\partial\Omega}\abs{u_m}\,d\sigma}.\] Moreover if $\Omega$ is connected $u_m$ is unique.
\end{teor}
\begin{proof}
    As observed, problem \eqref{minE} is equivalent to \eqref{aux2}. Let 
    \[F(v)=\int_\Omega \abs{\nabla v}^2\,dx+\dfrac{1}{m}\left(\int_{\partial\Omega} \abs{v}\right)^2-2\int_\Omega fv\,dx,\]
    and let $v_k\in H^1(\Omega)$ be a minimising sequence for problem \eqref{aux2}. Without loss of generality, we can assume
    \[F(v_k)\le F(0)=0,\]
    so that, using H\"{o}lder's and Young's inequalities on the source term
    \[2\int_\Omega fv_k\,dx,\]
    we have that, for every $\delta>0$, there exists $C_\delta>0$ such that
    \[\int_\Omega \abs{\nabla v_k}^2\,dx+\dfrac{1}{m}\left(\int_{\partial\Omega} \abs{v_k}\right)^2\le C_\delta+\delta\int_\Omega v_k^2\,dx.\]
    By the Poincarè-type inequality \eqref{pb}, we have that
    \[\int_\Omega \abs{\nabla v_k}^2\,dx+\dfrac{1}{m}\left(\int_{\partial\Omega} \abs{v_k}\right)^2\le C_\delta+\delta C\left[\int_\Omega \abs{\nabla v_k}^2\,dx+\left(\int_{\partial\Omega} \abs{v_k}\right)^2\right],\]
    so that $v_k$ is bounded in $H^1(\Omega)$. Up to a subsequence $v_k$ converges, weakly in $H^1(\Omega)$, to a function $u_m\in H^1(\Omega)$. By the $H^1(\Omega)$-lower semicontinuity of $F$
    we have that $u_m$ is a solution to \eqref{aux2}.\medskip

    Assume $\Omega$ is connected. To prove the uniqueness of the solution, let $u_1,u_2\in H^1(\Omega)$ be solutions to \eqref{aux2}. By direct computations, we have 
    \begin{equation}\label{convexityineq}\begin{split}F\left(\dfrac{u_1+u_2}{2}\right)-\dfrac{F(u_1)+F(u_2)}{2}=&-\displaystyle\int_\Omega \abs*{\dfrac{\nabla u_1-\nabla u_2}{2}}^2\,dx+\dfrac{1}{m}\left(\int_{\partial\Omega}\abs*{\dfrac{u_1+u_2}{2}}\,d\sigma\right)^2+\\[15pt]&-\dfrac{1}{2m}\left(\int_{\partial\Omega}\abs{u_1}\,d\sigma\right)^2-\dfrac{1}{2m}\left(\int_{\partial\Omega}\abs{u_2}\,d\sigma\right)^2.
    \end{split}\end{equation}
    By the convexity of the boundary term
    \[v\mapsto\dfrac{1}{m}\left(\int_{\partial\Omega}\abs{v}\,d\sigma\right)^2,\]
    the right-hand side of \eqref{convexityineq} is strictly negative unless $u_1-u_2$ is constant. On the other hand, by the minimality of $u_1$ and $u_2$, the left-hand side of \eqref{convexityineq} is non-negative. Hence, $u_1=u_2+c$. Similarly, if $u_1$ and $u_2$ have different sign on a subset $\Gamma$ of $\partial\Omega$ with $\Hn(\Gamma)>0$, we have
    \[\abs{u_1+u_1}<\abs{u_1}+\abs{u_2}\]
    on $\Gamma$ and the right-hand side of \eqref{convexityineq} is strictly negative. Hence $u_1$ and $u_2$ have the same sign and
    \[\begin{split}-\dfrac{c^2}{4m}\Hn(\partial\Omega)=&
    -\dfrac{1}{m}\left(\int_{\partial\Omega}\abs*{\dfrac{u_1-u_2}{2}}\,d\sigma\right)^2=
    \dfrac{1}{m}\left(\int_{\partial\Omega}\abs*{\dfrac{u_1+u_2}{2}}\,d\sigma\right)^2+\\[15pt]&-\dfrac{1}{2m}\left(\int_{\partial\Omega}\abs{u_1}\,d\sigma\right)^2-\dfrac{1}{2m}\left(\int_{\partial\Omega}\abs{u_2}\,d\sigma\right)^2, \end{split}\]
    which, by \eqref{convexityineq} and the minimality of $u_1$ and $u_2$, is non-negative,
    so that $c=0$ and $u_1=u_2$.
\end{proof}

Similar optimisation results have been obtained in \cite{BZ97} for the limit problem related to the clamped plate problem, in \cite{DPNST21} for the limit problem for the reinforcement problem with Robin boundary conditions, in \cite{EJEAEMERS24} for the vectorial case, and in \cite{dS25} for the $p$-Laplacian.\medskip

In \cite{ER03}, the authors studied the asymptotic behaviour of the optimal configuration of the insulating material "when $m$ is infinitesimal", proving that "the insulator has to be put in the points where the dispersion is maximal". Namely, let $u_0$ be the unique solution to the Dirichlet boundary value problem
\[\begin{cases}
    -\Delta u_0 = f &\text{in }\Omega,\\[5pt]
    u_0=0 &\text{on }\partial\Omega, 
\end{cases}\]
representing the steady-state temperature in the \emph{uninsulated} setting, and let 
\[K=\Set{\sigma\in\partial\Omega\colon\,\abs{\partial_\nu u_0}(\sigma)= \displaystyle\max_{\partial\Omega} \abs{\partial_\nu u_0}},\]
the set of points where the normal derivative, hence the heat flux, is maximal, then they proved the following theorem
\begin{teor}
    The sequence $h_m/m$ converges weakly-* in the sense of Radon measures, as $m$ goes to zero, to a finite, positive, Radon measure $\lambda$  supported on the set $K$.
\end{teor}\medskip

In \cite{BBN17}, the authors proposed the study of the \emph{shape optimisation} problem that arises from \eqref{minE} when $\Omega$ is allowed to vary. Namely, for a bounded open (sufficiently smooth) set $\Omega$ denote by $E(\Omega)$ the minimum of problem \eqref{minE} on $\Omega$, and consider the problem 
\begin{equation}\label{minEshape}\min_{\abs{\Omega}\le V_0}E(\Omega).\end{equation}
Moreover, when $f\equiv1$, in analogy with the famous Saint-Venant inequality, the authors conjectured that the solution to problem \eqref{minEshape} is the ball.

A partial solution to the problem has been proven in \cite{DLW21}. Namely, the authors proved the existence of a solution for problem \eqref{minEshape} in the class
of $M$-uniform domains, contained in a ball, of given measure and equi-bounded perimeter. If the source term $f$ is non-negative and in $L^n(\R^n)$, the authors also proved existence in the weaker setting in which one identifies the set $\Omega$ with the support of an $\sbv$ function. 

Finally, the case $f\equiv1$ has been studied in \cite{DPNST21}, proving that the solution to \eqref{minEshape} is indeed the ball.

\subsection{Spectral optimisation}
Let $\Omega\subset\R^n$ be a bounded Lipschitz set, and denote by $\lambda(h)$ the principal eigenvalue of problem \eqref{modeig} on $\Omega$. We have the following variational characterisation
\[\lambda(h)=\min_{v\in H^1(\Omega)\setminus\set{0}}\dfrac{\displaystyle\int_\Omega \abs{\nabla v}^2\,dx+\int_{\partial\Omega} \dfrac{v^2}{h}\,d\Hn}{\displaystyle\int_\Omega v^2\,dx}.\]
We are interested in the minimisation problem 
\begin{equation}\label{minEig}
    \lambda_m(\Omega):=\min\Set{\lambda(h)\colon\,h\in \mathcal{H}_m},
\end{equation}
studied in \cite{BBN17}, we also refer the reader to \cite{DPO25} for the minimisation of the first eigenvalue of the reinforcement problem with Robin boundary conditions, and to \cite{dS25} for the case of the $p$-Laplacian.\medskip

Let 
\[J_m(v)=\dfrac{\displaystyle\int_\Omega \abs{\nabla v}^2\,dx+\dfrac{1}{m}\left(\int_{\partial\Omega} \abs{v}\,dx\right)^2}{\displaystyle\int_\Omega v^2\,d\Hn}.\]
Arguing as for problem \eqref{minE}, we have that problem \eqref{minEig} is equivalent to the auxiliary problem
\begin{equation}\label{minEigaux}
    \min_{v\in H^1(\Omega)\setminus\set{0}} J_m(v),
\end{equation}
which, by the Poincaré-type inequality \eqref{pb} and the direct methods of the calculus of variations,  admits a non-negative solution $\bar{u}_m$. We have the following theorem
\begin{teor}
    Let $\bar{u}_m$ be a solution to problem \eqref{minEigaux} and let 
    \[\bar{h}_m=m\dfrac{\bar{u}_m}{\displaystyle\int_{\partial\Omega}\bar{u}_m\,d\Hn}.\]
    Then $\bar{h}_m$ is a solution to problem \eqref{minEig}, and the couple $(\bar{u}_m,\bar{h}_m)$ is a solution to \eqref{modeig}
\end{teor}

We have the following properties for the minimum eigenvalue as a function of $m>0$ 
\begin{prop}\label{proplambdam}
    Let $\Omega\subset\R^n$ be a bounded Lipschitz set. Then the function $m\in(0,+\infty)\mapsto\lambda_m(\Omega)$ is continuos and strictly decreasing, moreover 
    \[\lim_{m\to+\infty}\lambda_m(\Omega)=0,\]
    \[\lim_{m\to0^+}\lambda_m(\Omega)=\lambda^D(\Omega),\]
    the first Dirichlet eigenvalue of $\Omega$.
\end{prop}

\subsubsection*{Symmetry-breaking}

When $\Omega$ is a ball, an interesting phenomenon has been observed in \cite{BBN17} (see also \cite{DPO25}), namely minimisers of $J_m$ are not always radial functions, hence optimisers $\bar{h}_m$ to problem \eqref{minEig} can be non-constant (see \autoref{fig4}). This surprising phenomenon is called \emph{symmetry-breaking} and is contained in the following result.

\begin{figure}
\centering
\begin{tabular}{@{}c@{}}

    \includegraphics[height=0.5\linewidth]{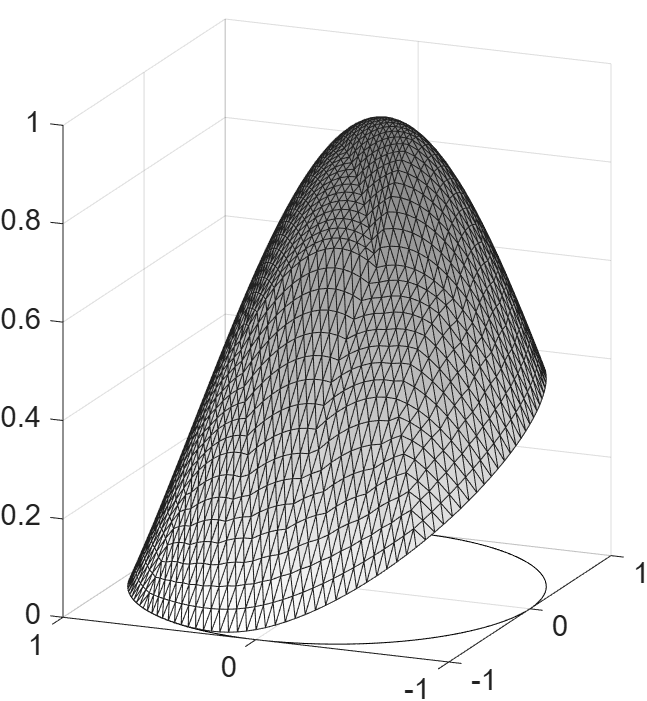}
\end{tabular}\qquad
\begin{tabular}{@{}c@{}}
    \centering
        \includegraphics[height=0.4\linewidth]{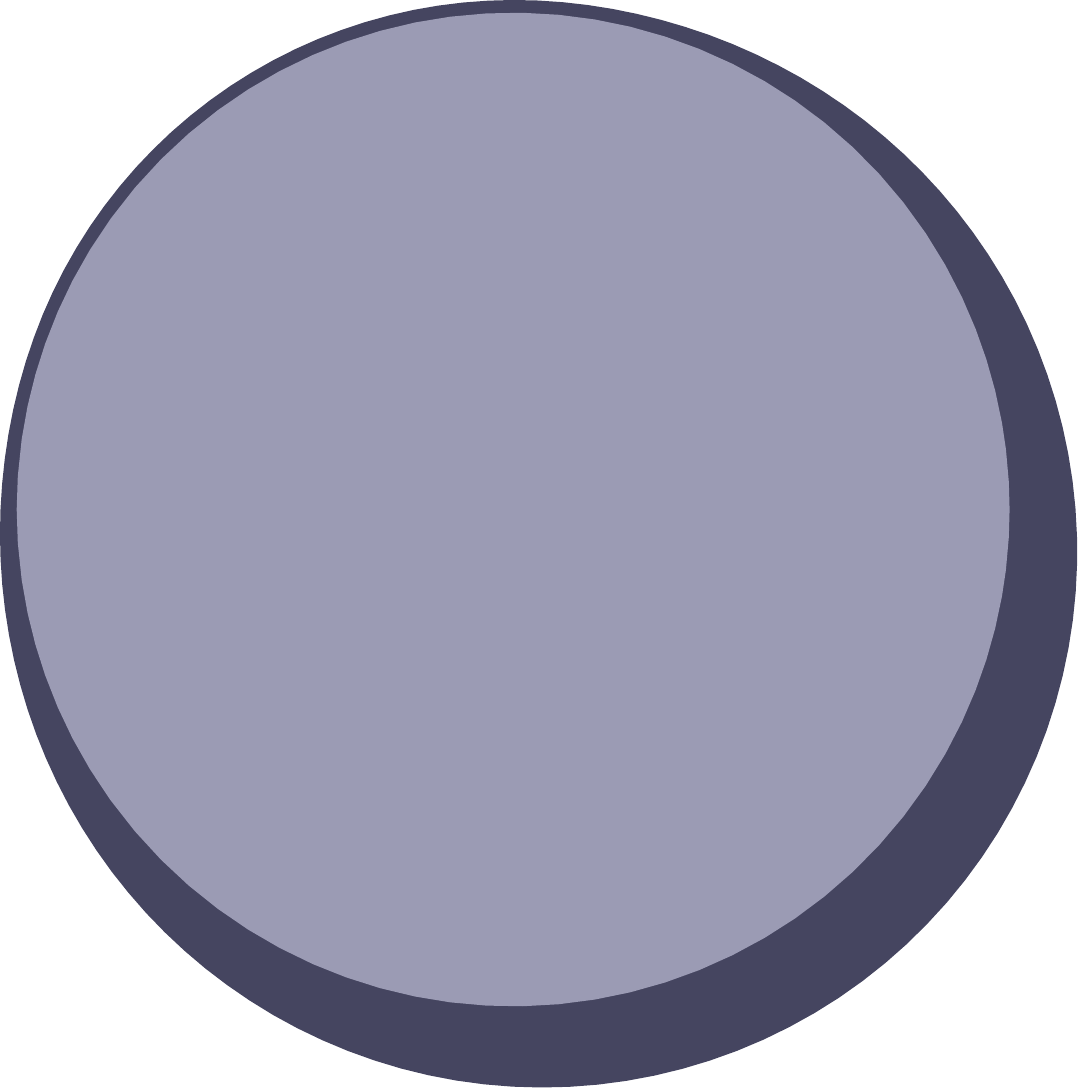}
\end{tabular}
 \caption{Example of non-radial solution to \eqref{minEigaux} and corresponding configuration of insulating material.}
    \label{fig4}
\end{figure}

\begin{teor}\label{sym-break}
    Let $\Omega=B_R$ be a ball of radius $R>0$ in $\R^n$ with $n\ge2$. There exists $m_0>0$ such that the solutions to the auxiliary minimisation problem \eqref{minEigaux} are
    \begin{itemize}
        \item[i)] radial, if $m>m_0$,
        \item[ii)] non-radial, if $0<m<m_0$.
    \end{itemize}
    In particular $\bar{h}_m$ is non-constant if $m<m_0$.
\end{teor}
\begin{proof}
    Let $\mu(B_R)$ be the principal eigenvalue of the Neumann problem on $B_R$ an let $\bar{v}$ be an associated $L^2(B_R)$-normalised eigenfunction, then $0<\mu(B_R)<\lambda^D(B_R)$ hence, by proposition \autoref{proplambdam}, there exists a unique $m>0$, such that
    \[\lambda_m(B_R)=\mu(B_R).\]
    Let $m_0$ be such a value. For every $m>0$, let $\bar{u}_m$ be a non-negative solution to \eqref{minEigaux} with normalised $L^2$-norm. Notice that, by the minimum principle, $\bar{u}_m$ can only vanish on the boundary of $B_R$ and that the trace of $\bar{u}_m$ cannot be identically equal to zero as $\lambda_m(B_R) < \lambda^D(B_R)$. In particular, if $u_m$ is radial, then it is strictly positive. Conversely, if $m\ne m_0$, the only positive solutions are the radially symmetric ones. Indeed, a solution to \eqref{minEigaux} is strictly positive if and only if it is a strictly positive solution to
    \begin{equation}\label{u>0}\begin{cases}-\Delta u=\lambda_m(B_R)u &\text{in }B_R,\\[7pt]
    \partial_\nu u = -\dfrac{1}{m}\displaystyle\int_{\partial B_R}u\,d\Hn&\text{on }\partial B_R.\end{cases}
     \end{equation}
     In particular, given a positive solution to \eqref{u>0}, we can construct a positive and radial solution by considering the average along the angular coordinates. On the other hand, any two positive solutions to \eqref{u>0} can be combined linearly to obtain a solution to the homogeneous  Neumann eigenvalue problem, so that positive solutions to \eqref{u>0} are unique up to a multiplicative constant.  \medskip
    
    We begin by proving that a strictly positive solution can exist only if $m\ge m_0$, so that, in particular, if $m<m_0$, no radial solution exists.  Let $m\ne m_0$ so that $\lambda_m(B_R)\ne\mu(B_R)$, and assume that $\bar{u}_m$ is strictly positive, than it satisfies 

    \[\int_{B_R} \nabla \bar{u}_m\nabla\varphi\,dx+\dfrac{1}{m}\int_{\partial B_R}\bar{u}_m\,d\Hn\int_{\partial B_R}\varphi\,d\Hn=\lambda_m(B_R)\int_{B_R} \bar{u}_m\varphi\,dx,\]
    for every $\varphi\in H^1(B_R)$. We have that $\bar{u}_m$ and the Neumann eigenfunction $\bar{v}$ are orthogonal, indeed
    \[(\mu(B_R)-\lambda_m(B_R))\int_{B_R} \bar{u}_m\bar{v}\,dx=-\dfrac{1}{m}\int_{\partial B_R} \bar{u}_m\,d\Hn\int_{\partial B_R} \bar{v}\,d\Hn=0,\]
    where we used that, since $B_R$ is a ball
    \[\int_{\partial B_R}\bar{v}\,d\Hn=0.\]
    As $\bar{u}_m$ is strictly positive the function $\varphi_\delta = \bar{u}_m+\delta \bar{v}$ is strictly positive for $\delta$ sufficiently small, so that using $\varphi_\delta$ as a test function for $\lambda_m(B_R)$ and using the orthogonality of $\bar{u}_m$ and $\bar{v}$, we have 
    \[\lambda_m(B_R)\le J_m(\varphi_\delta)=\dfrac{\lambda_m(B_R)+\delta^2\mu(B_R)}{1+\delta^2},\]
    that is
    \[\lambda_m(B_R)\le \mu(B_R)\]
    and, by monotonicity, $m>m_0$.\medskip

    We now need to prove that non-radial solutions only exist if $m\le m_0$. Let $m\ne m_0$ and assume that a non-radial solution exists. Let $\bar{u}_m(r,\theta)$ be such a solution in polar coordinates. By spherical symmetrisation (see \cite{S81, D99}) we can assume that $\bar{u}_m$ is \emph{spherically symmetric} in a direction $\theta_0\in\mathbb{S}^{n-1}$, that is
    \[\bar{u}_m(r,\theta_1)\ge\bar{u}_m(r,\theta_2)\quad\text{for every }0<r<R\,\text{ and }\abs{\theta_1-\theta_0}\le\abs{\theta_2-\theta_0}.\]
    Indeed, by the properties of the rearrangement, the integral of the gradient decreases while the $L^2$-norm and the boundary integral remain the same. Moreover $\bar{u}_m$ satisfies
\begin{equation}\label{u>=0}\begin{cases}-\Delta \bar{u}_m=\lambda_m(B_R)\bar{u}_m &\text{in }B_R,\\[7pt]
    \partial_\nu \bar{u}_m = -\dfrac{1}{m}\displaystyle\int_{\partial B_R}\bar{u}_m\,d\Hn&\text{on }\partial B_R\cap\set{\bar{u}_m>0},\\[10pt]
    \partial_\nu \bar{u}_m \ge -\dfrac{1}{m}\displaystyle\int_{\partial B_R}\bar{u}_m\,d\Hn&\text{on }\partial B_R\cap\set{\bar{u}_m=0}.\end{cases}
     \end{equation}

     Multiplying equation \eqref{u>=0} by the Neumann eigenfunction $\bar{v}$ and integrating by parts we have
     \begin{equation}\label{last}
         (\mu(B_R)-\lambda_m(B_R))\int_{B_R} \bar{u}_m\bar{v}\,dx=\int_{\partial B_R}\partial_\nu \bar{u}_m\bar{v}\,d\Hn.
     \end{equation}
     Without loss of generality we can also assume that $\bar{v}$ is spherically symmetric in the direction $\theta_0$, as a consequence $\bar{v}$ is antisymmetric in the direction orthogonal to $\theta_0$ and, as $\bar{u}_m$ is spherically symmetric in direction $\theta_0$ and non-constant on $\partial B_R$ 
     \[\int_{B_R} \bar{u}_m\bar{v}\,dx>0. \]
     Finally, by the symmetry of $\bar{u}_m$ and the boundary conditions in \eqref{u>=0}, we can deduce that also $-\partial_\nu \bar{u}_m$ is spherically symmetric in direction $\theta_0$, hence the right hand side of \eqref{last} is non positive and $\lambda_m(B_R)\ge\mu(B_R)$, that is $m<m_0$. 
\end{proof}
In \cite{HLL22} (see also \cite{HLL24}) the exact value of $m_0$ was computed to be
\[m_0=\dfrac{n-1}{n}\dfrac{P(B_R)^2}{\abs{B_R}\,\mu(B_R)},\]
where $P(B_R)$ is the perimeter of $B_R$.\medskip

In \cite{BBN17}, the authors observed that, in general, the shape optimisation problem
\[\min\Set{\lambda_m(\Omega)\colon\,\abs{\Omega}= V_0}\]
is ill-posed if the dimension, $n$, is larger or equal than $3$. Indeed, if $\Omega$ is a ball of radius $R$, using the constant as a test function, we have
\[\lambda_m(B_R)\le J_m(B_R)=\frac{n^2 \omega_n  }{m}R^{n-2},\]
where $\omega_n$ is the volume of the unit ball. Hence, if $n\ge3$, $\lambda_m(B_R)$ goes to zero as $R$ goes to zero. Then, if $\Omega_k$ is the disjoint union of $k$ balls of measure $V_0/k$, we have
\[\lim_{k\to+\infty} \lambda_m(\Omega_k)=0.\]
Moreover, for $n=2$, they proved that balls are stationary, in the sense of shape variations, among sets of fixed volume (when $m\ne m_0$) if and only if $m>m_0$. The problem has also been recently studied in \cite{BBK25} in dimensions $n=1,3$ among convex sets of given volume, adding a lower bound constraint on the functions in $\mathcal{H}_m$, proving the existence of an optimal shape.\medskip

 In \cite{HLL24} (see also \cite[Theorem 1.9]{HLLY24}) the following generalisation to \autoref{sym-break} was proved

\begin{teor}\label{sym-break2}
    Let $\Omega\subset\R^n$ be a bounded Lipschitz domain. Then there exists $m_0>0$ such that, when $m<m_0$, the solutions to the auxiliary minimisation problem \eqref{minEigaux} vanish on a subset of the boundary with positive $\Hn$-measure.
\end{teor}
The proof of the result is analogous to the first part of the proof of \autoref{sym-break}, where instead of the Neumann eigenvalue, they considered
\[k_1(\Omega)=\min\Set{\dfrac{\displaystyle\int_\Omega \abs{\nabla v}^2\,dx}{\displaystyle\int_\Omega v^2\,dx}\colon\,v\in H^1(\Omega)\setminus\set{0},\,\int_{\partial\Omega}v\,d\Hn=0}.\]
Finally, the authors conjectured that, denoting by $m_0(\Omega)$ the largest value of $m_0$ such that \autoref{sym-break2} holds, then 
\[m_0(\Omega)\ge m_0(\Omega^*),\]
where $\Omega^*$ is the ball of the same volume. That is, among sets of given volume, the ball should be the set for which the symmetry-breaking phenomenon is less likely to happen. To our knowledge, such a problem remains open. 

\subsubsection*{Acknowledgements} 
We would like to thank Prof. Dr. S\"{o}eren Bartels for providing us with the routines needed to generate \autoref{fig4}.\medskip

The three authors are members of Gruppo Nazionale per l’Analisi Matematica, la Probabilità e le loro Applicazioni (GNAMPA) of Istituto Nazionale di Alta Matematica (INdAM).

\printbibliography[heading=bibintoc]

\end{document}